\documentclass[11pt]{amsart}
\usepackage{amssymb}
\usepackage{amsmath}
\usepackage{mathrsfs}
\usepackage{amsfonts}
\usepackage{color}
\usepackage{amsthm}
\usepackage{graphicx}
\usepackage[colorlinks=true, pdfstartview=FitV, linkcolor=blue,citecolor=blue, urlcolor=blue,dvipdfmx]{hyperref}
\usepackage{verbatim}

\usepackage[utf8]{inputenc}
\usepackage[T1]{fontenc}
\numberwithin{equation}{section}

\theoremstyle{definition}
\newtheorem{defi}{Definition}[section]
\newtheorem{rem}[defi]{Remark}
\theoremstyle{plain}
\newtheorem{theorem}{Theorem}[section]

\newtheorem{lem}[defi]{Lemma}
\newtheorem{cor}[defi]{Corollary}
\newtheorem{prop}[defi]{Proposition}

\newenvironment{pr}[1]
   {{\noindent \bf Proof of {#1}.\  }}{\hfill \qed}

\numberwithin{equation}{section}

\newcommand{\nn}{\nonumber}
\newcommand{\io}{\int_0^1}

\newcommand{\dx}{\partial_x}

\newcommand{\eqntag}{\addtocounter{equation}{1}\tag{\theequation}}
\allowdisplaybreaks

\renewcommand{\r}{\mathbb{R}}

\author{Tomasz Cie\'{s}lak}
\address{Institute of Mathematics \newline Polish Academy of Sciences \newline \'Sniadeckich 8, 00-656 Warszawa, Poland}
\email{cieslak@impan.pl}

\author{Kentaro Fujie}
\address{Mathematical Institute, Tohoku University, Sendai 980-8578, Japan}
\email{fujie@tohoku.ac.jp}

\author{Tatsuya Hosono}
\address{Osaka Central Advanced Mathematical Institute, Osaka Metropolitan University,
Osaka 558-8585, Japan}
\email{tatsuya.hosono@omu.ac.jp}

\title[Nonlinear Fisher information]{Nonlinear Fisher information, corresponding functional inequalities and applications}

\begin{document}
\begin{abstract}
We study the evolution of the nonlinear version of the Fisher information along the quasilinear heat equation. We also provide a nonlinear version of a recent functional inequality (\cite{Ci-Ma-Ka-Mi}), corresponding to the nonlinear heat equation. Next, applications of our version of nonlinear Fisher information to the 1D critical quasilinear fully parabolic Keller--Segel system are given. In particular, the global existence of solutions to the critical nonlinear diffusion/nonlinear sensitivity 1D fully parabolic Keller--Segel system is obtained for certain type of diffusion. Last, but not least, we also study the version of the Fisher information along the $p$-Laplace equation.
\end{abstract}

\maketitle

\vspace{5mm}

\vspace{5mm}


\section{Introduction}
Diffusion phenomena are often governed by nonlinear partial differential equations, where entropy and Fisher information
play an important role, see~\cite{Ju}. The study of entropy and Fisher information is  important in sciences, like physics (\cite{Ev,Ju,Li-Ya24}) or biology (\cite{Ju}), as well as information theory~(\cite{Sa-To,Sh}).  In mathematics, they provide powerful analytical tools for proving large time behavior, regularity or asymptotic behavior of solutions~(\cite{Ca-Ju-Ma-To-Un,CaTo}).

The role of entropy and its monotonicity has been widely explored in both linear and nonlinear settings. The corresponding behavior of Fisher information, particularly its time monotonicity, is also an important tool in the studies of diffusion phenomena, see~\cite{Bakry,Ca-Ju-Ma-To-Un} or due to its applications to information theory, see \cite{Sa-To}. In particular, the monotonicity of nonlinear version of Fisher information in the case of porous medium equations, posed in a whole space has been given in \cite{Sa-To}.

In the present paper we shall study a slightly different approach to the nonlinear Fisher information. However, in the beginning we shall extend the classical approach by Savar\'e--Toscani to the nonlinear diffusions in bounded domains with zero Neumann boundary conditions.

Next, we investigate an alternative formulation of the Fisher information to quite general nonlinear diffusion equations under homogeneous Neumann boundary conditions, focusing on the evolution of Fisher information.
We show that, under suitable structural assumptions, the Fisher information exhibits time-monotonicity. Next, we establish related novel functional inequalities that allow to use the dissipative term in certain types of
systems of PDEs. Finally, we shall apply our approach to obtain in a shorter and more direct way a functional that was a main tool that led to the proof of the global existence result in the quasilinear critical 1D nonlinear diffusion Keller--Segel system, see \cite{Ci-Fu18}, see also \cite{BCFS, CF1, Fu}.  Moreover, we shall also use our method to give a completely new result concerning the global existence in the critical 1D Keller--Segel system with nonlinear diffusion and nonlinear sensitivity. Last, but not least, a new Fisher information type quantity exhibiting the time-monotonicity in the $p$-Laplace equation in 1D will be given.

Let $\Omega\subset \r^n$ be a bounded domain with a smooth boundary $\partial\Omega$.
We study the initial boundary value problem of the following nonlinear diffusion equations under homogeneous Neumann boundary condition:
\begin{equation}
\left\{
\begin{aligned}
&u_t  = \nabla\cdot(a(u)\nabla u),
& t>0,\, &x\in\Omega,
\\
&\frac{\partial u}{\partial \nu}=0,
&t>0,\, &x\in\partial\Omega,
\\
&u(0,x)=u_{0}(x),
&\, &x\in\Omega,
\end{aligned}
\right.
\label{eqn;NDE}
\end{equation}
where the initial data being such that $u_0\ge 0$ in $\Omega$,
$\nu$ denotes the outward unit normal vector on $\partial\Omega$,
and
the positive function $a$ on $[0,\infty)$ satisfies
\begin{align*}
a(s)\ge0,
\quad
a\in C^1((0,\infty)).
\label{eqn;a}
\end{align*}
Typical choice of the function $a(s)$ is represented as
\begin{align*}
a(s) = m(1+s)^{m-1}
\end{align*}
with $m\in\mathbb{R}$. Typical candidates of entropies in the particular case $a(s)=m s^{m-1}$
are given by
\begin{align*}
&H_m(u):=\frac{1}{m-1}\int_{\Omega} u^m\,dx,\quad m>0,~~m\neq 1,
\\
&H_1(u):=\int_{\Omega}u(\log u -1)\, dx,
\\
&H_0(u):=\int_{\Omega}(u-\log u)\,dx.
\end{align*}
The entropy $H_1(u)$ is known as the (Boltzmann--)Shannon entropy \cite{Sh} for the linear heat, and satisfies
\begin{align*}
\frac{d}{d t} H_1(u(t)) = -\int_{\Omega} \frac{|\nabla u |^2}{u}\,dx\le\,0,
\end{align*}
where the negative of the right hand side is called the Fisher information (entropy production) \cite{Fi}, first used by Edgeworth \cite{Ed}, that plays an important role in information theory (or statistics).

Such a classical Fisher information has recently found applications in the 1D thermoelasticity on the one hand, see \cite{Bi-Ci23, Bi-Ci25}, where global-in-time existence and uniqueness of solutions (see \cite{Bi-Ci23}) as well as the full asymptotic classification of solutions (\cite{Bi-Ci25}) were obtained using the functional built on the Fisher information.
In the higher dimensional thermoelasticity such methods were used in \cite{Bi-Ci-La-Mu-Tr}.
Next, application of the above method was successful also in the combustion theory, see \cite{Li-Ya24}.
In the above mentioned articles the functional inequality introduced in \cite{Ci-Ma-Ka-Mi} was used.
The inequality states that there exists a constant
$C>0$ such that for all functions
\[
f\in {\bf F}:=
\left\{
f:\Omega\rightarrow \r_+,~
\begin{split}
&\mbox{of $C^2$ regularity, satisfying}
	\\& \mbox{zero Neumann data at the boundary}
\end{split}
\right\},
\]
\begin{equation}\label{wstepne}
\int_\Omega |D^2 \sqrt{f}|^2dx\leq C \int_\Omega f|D^2 \log f|^2dx.
\end{equation}
The applicability of the above inequality stems from the fact that along the heat flow, the Fisher information satisfies
\begin{equation}\label{wstepne_bis}
\frac{d}{dt}\int_\Omega \frac{|\nabla u|^2}{u}dx=-\int_\Omega u|D^2 \log u|^2dx,
\end{equation}
so that the dissipation term of the Fisher information is given by the negative of the right-hand side of \eqref{wstepne}.

In the case of a porous medium equation, at least in the whole space the similar result is known, see \cite{Sa-To} and Section \ref{pierwsza}.
In the present paper we first extend the Savar\'e--Toscani result to other type of nonlinear diffusions and bounded domains with homogeneous Neumann data.
Next,
this is the main contribution of the paper, we present an alternative approach to the problem.
It gives the monotonicity of the nonlinear Fisher information for a very general class of nonlinear diffusions in 1D, some results in higher dimensions are also available.
For the details, we refer the reader to Section \ref{druga}.
Next, we introduce the class of nonlinear inequalities corresponding to the nonlinear Fisher information dissipation term, similar in its spirit to \eqref{wstepne}.
Moreover, the nonlinear versions of the Bernis type inequalities, different from those obtained in \cite{Wi12}, are also presented.
The latter two classes of inequalities are found in Section \ref{trzecia}.
In Section \ref{czwarta} we give two applications of our method.
First, in Section \ref{4.5} a functional known from \cite{Ci-Fu18} is derived in a more direct way.
This proof is based on our new approach introduced in Section \ref{druga}. The functional plays a crucial role in proving the existence of global solutions in the critical 1D nonlinear diffusion Keller--Segel system.
Next, a particular case of nonlinear diffusion and nonlinear sensitivity 1D Keller--Segel system is considered in Section \ref{4.66}.
We show that independently on the magnitude of initial data, the solution exists globally in time.
Again, we use a variant of our method to obtain the crucial estimate. We remark that the result is completely new and the well-known method based on Lyapunov functional to the Keller--Segel system is useless in our setting, see Remark \ref{poor}.
Similarly to the nonlinear diffusion case, it shows that 1D Keller--Segel system (for both nonlinear diffusion and sensitivity) does not allow a critical mass phenomenon,
unlike in higher dimensions. Last, but not least, in Section~\ref{piata}, a variant of the nonlinear Fisher information, corresponding to $p$-Laplace equation is introduced.
We show that it is monotone along the flow in space dimension 1. The latter result also seems new.

\section{Savar\'e--Toscani calculation in a bounded domain}\label{pierwsza}
In this section we extend the calculation in \cite{Sa-To} to the case of a bounded domain and homogeneous Neumann data.
Our result covers also quite a wide range of nonlinearities.

In the following theorem, the distinction between the Laplacian and the Hessian becomes critical when one attempts to establish monotonicity
properties of functionals along the evolution governed by nonlinear diffusion equations \eqref{eqn;NDE} in higher dimensions.

This section should be viewed as small step.
It is only a slight extension of the calculations in \cite{Sa-To} to the bounded domain case.
\begin{theorem}\label{thm;FI}
For $n\ge1$, let $\Omega\subset\r^n$ be a smooth and  bounded domain.
Let $u $ be a positive classical solution to \eqref{eqn;NDE} in $ C^3(\bar{\Omega})$.
Define primitive functionals $\Lambda(s)$, $H(s)$ and $F(s)$ on $[0,\infty)$ by
\begin{align*}
	\Lambda(s):=\int_1^s \frac{a(\tau)}{\tau}d\tau,
	\quad
	H(s):=\int_1^s \Lambda(\tau)d\tau,
	\quad
	F(s):=\int_{0}^s a(\tau)d\tau.
	\eqntag\label{eqn;premitive-higher}
	\end{align*}
	Then the following holds true:
	\begin{align*}
	\frac{d}{dt} \int_{\Omega} H(u)  dx =\,& -\int_{\Omega} u |\nabla \Lambda(u)|^2dx,
	\end{align*}
	and
	\begin{align*}
	&\frac{1}{2}\frac{d}{dt}\int_{\Omega} u |\nabla \Lambda(u)|^2 dx
	\\= \,&
	-\int_{\Omega} F(u) |D^2 \Lambda(u)|^2dx-\int_{\Omega}(ua(u) - F(u)) |\Delta \Lambda(u)|^2dx
	\\&+\frac12\int_{\partial\Omega}F(u)\nabla |\nabla \Lambda(u)|^2\cdot \nu dS.
	\end{align*}
If $\Omega$ is convex then the boundary term is non-positive, so that
\begin{align*}
&\frac{1}{2}\frac{d}{dt}\int_{\Omega} u |\nabla \Lambda(u)|^2 dx
\\\le \,&-\int_{\Omega} F(u) |D^2 \Lambda(u)|^2dx-\int_{\Omega}(ua(u) - F(u)) |\Delta \Lambda(u)|^2dx.
\end{align*}
\end{theorem}
In order to complete the proof of Theorem \ref{thm;FI}, we
recall the following two useful lemmata.
\begin{lem}[Bochner's formula]\label{lem;Bochner}
Let $\Omega$ be a smooth domain of $\r^n$ and $f \in C^3(\bar{\Omega})$.
Then,
\begin{align*}
\nabla f \cdot \nabla(\Delta f)=- |D^2 f|^2+\frac{1}{2} \Delta |\nabla f|^2
\quad
\text{in}\,\,\bar{\Omega}.
\end{align*}
\end{lem}
\begin{lem}\label{lem;Evans}
Let $\Omega$ be a convex bounded domain of $\r^n$ with smooth boundary~$\partial\Omega$.
Suppose that a function $f \in C^2(\bar{\Omega})$ satisfies the $0$-Neumann boundary condition. Then,
\begin{align*}
\left.\frac{\partial |\nabla f|^2}{\partial \nu}\right|_{\partial\Omega}\leq 0.
\end{align*}
\end{lem}
The proof of Lemma \ref{lem;Evans} can be found in \cite[Ch.4, B-2-b, 95 page]{Ev}.

It is worth mentioning that if $\Omega$ is convex and
 \begin{align*}
 ua(u) - F(u) \ge 0,
 \end{align*}
 then Theorem \ref{thm;FI} reads
	\[
	\frac{1}{2}\frac{d}{dt}\int_{\Omega} u |\nabla \Lambda(u)|^2 dx  \leq \,0.
	\]
The important application of Theorem \ref{thm;FI} is the following:
\begin{cor}[A straightforward example]\label{cor;monotonicity}
Let $\Omega$ be a convex smooth bounded domain in~$\r^n$.
Suppose assumptions as in Theorem \ref{thm;1D-FI}.
Let $a(u)= m u^{m-1}$ satisfying
\begin{align*}
1-\frac1n\le\, m.
\end{align*}
Then, the time-monotonicity of Fisher information is preserved along the trajectories of \eqref{eqn;NDE}, that is,
\begin{align*}
\frac{1}{2}\frac{d}{dt}\int_{\Omega} u |\nabla \Lambda(u)|^2 dx  \leq \,&0.
\end{align*}
\end{cor}
\begin{pr}{Corollary \ref{cor;monotonicity}}
It is clear how to proceed in the linear case $m=1$ and the degenerate diffusion case $m>1$.
Indeed, for $m\ge1$, it follows that $ua(u) - F(u) = (m-1)u^m$, so that
\begin{align*}
&\frac{1}{2}\frac{d}{dt}\int_{\Omega} u |\nabla \Lambda(u)|^2 dx
\\
=\,&-\int_{\Omega}u^m|D^2\Lambda(u)|^2dx-(m-1)\int_{\Omega}u^m|\Delta\Lambda(u)|^2dx
\\
&+\frac12\int_{\partial\Omega}F(u)\nabla |\nabla \Lambda(u)|^2\cdot \nu dS
\\
\le\,&-\int_{\Omega}u^m|D^2\Lambda(u)|^2dx-(m-1)\int_{\Omega}u^m|\Delta\Lambda(u)|^2dx\leq \,0,
\end{align*}
where we used Lemma \ref{lem;Evans}.

On the other hand, let $a(u)=mu^{m-1}$ with $1-1/n \le m < 1$.
Since $ua(u)=mu^m$ and $F(u)=u^m$, we see
\begin{align*}
&\frac{1}{2}\frac{d}{dt}\int_{\Omega} u |\nabla \Lambda(u)|^2 dx
\\
=\,&-\int_{\Omega}  u^m |D^2 \Lambda(u)|^2dx-\int_{\Omega}(m-1)u^m |\Delta \Lambda(u)|^2dx
\\&+\frac12\int_{\partial\Omega}F(u)\nabla |\nabla \Lambda(u)|^2\cdot \nu dS
\\
\le\,&-\int_{\Omega} u^m |D^2 \Lambda(u)|^2dx+n(1-m)\int_{\Omega}u^m |D^2\Lambda(u)|^2dx
\\
=\,&-(1-n+nm)\int_{\Omega} u^m |D^2 \Lambda(u)|^2dx\leq \,0,
\end{align*}
where we use the well known point-wise estimate
\begin{align*}
|\Delta f |^2 \le\, n |D^2 f |^2.
\end{align*}
Thus, the proof is complete.
\end{pr}
\vspace{3mm}

The formulation in the case $m>1$
corresponds to the result of
Savar\'e--Toscani~\cite{Sa-To} in the whole space $\r^n$.
\vspace{3mm}

We are now in a position to prove Theorem \ref{thm;FI}.
\vspace{3mm}

\begin{pr}{Theorem \ref{thm;FI}}
First recalling that
\begin{align*}
u_t=\nabla\cdot (a(u)\nabla u) =\nabla\cdot ( u \nabla \Lambda(u) ),
\end{align*}
we then have
\begin{align*}
\frac{d}{dt} \int_{\Omega}H(u)dx
=\,&\int_{\Omega} \Lambda(u) \partial_tu dx
\\
=\,&\int_{\Omega} \Lambda(u) \nabla\cdot ( u \nabla \Lambda(u) ) dx=\,-\int_{\Omega}u |\nabla \Lambda(u)|^2dx,
\end{align*}
where the $0$-Neumann boundary condition for $u$ ensures that
\begin{align*}
\nabla \Lambda(u)\cdot \nu=\frac{a(u)}{u}\nabla u \cdot \nu =0\quad
\text{on}~~\partial\Omega.
\end{align*}

Next, it follows from integration by parts that
\begin{align*}
&\frac{1}{2}\frac{d}{dt}\int_{\Omega}u |\nabla \Lambda(u)|^2 dx
\\
=\,&\frac{1}{2} \int_{\Omega} \partial_t u |\nabla \Lambda(u)|^2dx+\int_{\Omega}u\nabla \Lambda(u)\cdot\nabla \partial_t \Lambda(u)dx
\\
=\,&-\frac{1}{2} \int_{\Omega} u \nabla \Lambda(u)\cdot \nabla |\nabla \Lambda(u)|^2dx
\\
&+\int_{\Omega}u \nabla \Lambda(u)\cdot \nabla \left(\frac{a(u)}{u}\nabla \cdot \left[u\nabla \Lambda(u)\right]\right)dx
\\
=\,&-\frac{1}{2}\int_{\partial\Omega}F(u)\nabla |\nabla \Lambda(u)|^2\cdot \nu dS+\frac{1}{2}\int_{\Omega} F(u) \Delta |\nabla \Lambda(u)|^2dx
\\
&+\int_{\Omega}u \nabla \Lambda(u)\cdot \nabla \left(|\nabla \Lambda(u)|^2+a(u)\Delta \Lambda(u)\right)dx
\\
=\,&\frac{1}{2}\int_{\partial\Omega}F(u)\nabla |\nabla \Lambda(u)|^2\cdot \nu dS-\frac{1}{2}\int_{\Omega} F(u) \Delta |\nabla \Lambda(u)|^2dx
\\
&+\int_{\Omega} a(u)\nabla F(u) \cdot \nabla \Delta \Lambda(u)dx+\int_{\Omega}\nabla F(u)\cdot \nabla a(u) \Delta \Lambda(u)dx.
\end{align*}
Here, noting that
\begin{align*}
\nabla F(u) = \, a(u) \nabla u = \, u\nabla \Lambda(u),
\end{align*}
we see that
\begin{align*}
\int_{\Omega} a(u)\nabla F(u) \cdot \nabla \Delta \Lambda(u)dx=\,&\int_{\Omega} ua(u) \nabla \Lambda(u)\cdot \nabla \Delta \Lambda(u)dx.
\end{align*}
Additionally,
\begin{align*}
\Delta F(u) =\, \nabla \cdot \left(u\nabla \Lambda(u)\right) =\, \nabla u \cdot \nabla \Lambda(u)+u\Delta \Lambda(u),
\end{align*}
so that
\begin{align*}
&\int_{\Omega}\nabla F(u)\cdot \nabla a(u) \Delta \Lambda(u)dx
\\
=&-\int_{\Omega}a(u) \Delta F(u) \Delta \Lambda(u)dx-\int_{\Omega}a(u) \nabla F(u)\cdot \nabla \Delta \Lambda(u)dx
\\
=&-\int_{\Omega}\nabla F(u) \cdot \nabla \Lambda(u)\Delta \Lambda(u)dx-\int_{\Omega} ua(u) |\Delta\Lambda(u)|^2 dx
\\
&-\int_{\Omega}ua(u) \nabla \Lambda(u)\cdot \nabla \Delta \Lambda(u)dx
\\
=\,&\int_{\Omega}F(u) |\Delta \Lambda(u)|^2dx+\int_{\Omega} F(u)\nabla \Lambda(u)\cdot \nabla \Delta \Lambda(u)dx
\\
&-\int_{\Omega} ua(u) |\Delta \Lambda(u)|^2dx-\int_{\Omega}ua(u) \nabla \Lambda(u)\cdot \nabla \Delta \Lambda(u)dx,
\end{align*}
where
\begin{align*}
\nabla F(u)\cdot \nu =\, a(u) \nabla u \cdot \nu =0\,\quad\text{on}~~\partial\Omega.
\end{align*}
Collecting above estimates leads us to
\begin{align*}
&\frac{1}{2}\frac{d}{dt}\int_{\Omega}u |\nabla \Lambda(u)|^2 dx
\\
=\,&\frac{1}{2}\int_{\partial\Omega}F(u)\nabla |\nabla \Lambda(u)|^2\cdot \nu dS-\frac12\int_{\Omega} F(u) \Delta |\nabla \Lambda(u)|^2dx
\\
&+\int_{\Omega} F(u) \nabla \Lambda(u)\cdot \nabla \Delta \Lambda(u)dx-\int_{\Omega} \left( ua(u) - F(u)\right) |\Delta \Lambda(u)|^2 dx
\\
=\,&\frac{1}{2}\int_{\partial\Omega}F(u)\nabla |\nabla \Lambda(u)|^2\cdot \nu dS
\\
&-\int_{\Omega} F(u)|D^2 \Lambda(u)|^2dx-\int_{\Omega} \left( ua(u) - F(u)\right) |\Delta \Lambda(u)|^2dx,
\end{align*}
where we use the Bochner formula stated in Lemma \ref{lem;Bochner},
\[
-\frac{1}{2}\Delta|\nabla \Lambda(u)|^2+\nabla\Lambda(u)\cdot\nabla\Delta\Lambda(u)=-|D^2\Lambda(u)|^2.
\]
Finally, if $\Omega$ is convex, Lemma \ref{lem;Evans} implies that
\begin{align*}
\nabla |\nabla \Lambda(u)|^2\cdot \nu \leq \,0\quad
\text{on}~~\partial\Omega,
\end{align*}
consequently,
\begin{align*}
&\frac{1}{2}\frac{d}{dt}\int_{\Omega}u |\nabla \Lambda(u)|^2 dx
\\
\leq\,&-\int_{\Omega} F(u)|D^2 \Lambda(u)|^2dx-\int_{\Omega} \left( ua(u) - F(u)\right) |\Delta \Lambda(u)|^2 dx.
\end{align*}
Hence, the proof is complete.		
\end{pr}

\section{Alternative formulation of the nonlinear Fisher information}\label{druga}

This section deals with the introduction of the main idea of the paper.
We rewrite the Fisher information and derive the identity satisfied in our formulation along the flow of \eqref{eqn;NDE}.
Let us first introduce the necessary quantities to play main roles in our derivation.

As a consequence of our derivation, we shall show a simplified argument of a derivation of a Lyapunov functional that was a main tool in proving the global solutions to the one dimensional nonlinear diffusion critical Keller--Segel system, obtained in \cite{Ci-Fu18}, see Section \ref{4.5}.
Moreover, in Section \ref{4.66} we shall use our approach to solve the so far open problem if the solution of the 1D critical nonlinear diffusion/nonlinear sensitivity Keller--Segel system admits the critical mass.
Our result states that there is no critical mass phenomenon in the considered case.

\begin{theorem}\label{thm;1D-FI}
Let $a$ be a positive function.
For $\Omega=(0,1)$ and $T>0$,
let $u$ be the positive classical solution to problem \eqref{eqn;NDE},
set the primitive functionals $\Lambda(s)$, $H(s)$ and $\Sigma(s)$ on $(0,\infty)$ by
\begin{equation}
\Lambda(s):=\int_{1}^s \frac{a(\tau)}{\tau} d\tau,
\quad
H(s):= \int_1^s \Lambda(\tau) d\tau,	
\quad
\Sigma(s):=\int_1^s \frac{a(\tau)}{\sqrt{\tau}} d\tau.
\label{eqn;plimitive}
\end{equation}
Then, the solution to \eqref{eqn;NDE} satisfies:
\begin{align}
\frac{d}{dt}\int_{0}^1H(u) dx= - \int_{0}^1 |\partial_x \Sigma(u)|^2 dx.
\label{eqn;entropy}
\end{align}
Moreover, we have
\begin{align}
\frac{1}{2}\frac{d}{dt}\int_0^1|\partial_x \Sigma(u)|^2dx
=-\int_0^1ua(u)\left|\partial_x\left(\frac{1}{\sqrt{u}}\partial_x\Sigma(u)\right)\right|^2 dx.
\label{eqn;1D-Fisher}
\end{align}
\end{theorem}

\begin{rem}\label{rem;FI}
When $n=1$, then Theorem \ref{thm;FI} corresponds to Theorem \ref{thm;1D-FI}.
Indeed, let $\Omega=(0,1)$ and $\partial_x u(t,0) =\partial_x u(t,1)=0$.
By virtue of the fact that $|D^2 \Lambda(u) | = |\partial_{xx} \Lambda(u)|=|\Delta \Lambda(u)|$ for $n=1$, and
$\partial_x \Sigma(u) = \sqrt{u} \partial_x \Lambda(u)$, we have
\begin{align*}
\frac{1}{2}\frac{d}{dt}\int_{0}^1  |\partial_x \Sigma(u)|^2 dx
=\,&\frac{1}{2}\frac{d}{dt}\int_{0}^1 u |\partial_x \Lambda(u)|^2 dx
\\
=\,&-\int_{0}^1 u a(u) \left| \partial_x^2 \Lambda(u) \right|^2dx
\\
=\,&-\int_{0}^1 u a(u) \left| \partial_x\left(\frac{1}{\sqrt{u}}\partial_x \Sigma(u)\right) \right|^2dx.
\end{align*}
We shall give a straightforward proof of Theorem \ref{thm;1D-FI} below.
Although it also follows from Remark \ref{rem;FI}, presenting the proof here will be useful for our study of the 1D Keller--Segel system.
We also hope that the calculation can be applied to other types of PDE systems.
In fact, the calculation itself plays an important role when extending the method to systems of PDEs; see Section~\ref{czwarta}.
\end{rem}
\begin{rem}
Since $a(s)$ is positive, the function $t\mapsto H(u(t))$ is convex.
Indeed, \eqref{eqn;entropy} together with \eqref{eqn;1D-Fisher} yield
\begin{align*}
\frac{d^2}{dt^2} \int_{0}^1 H(u) dx=\,& - \frac{d}{dt}\int_{0}^1 \left|\partial_x\Sigma(u)\right|^2 dx
\\
=\,&2\int_0^1ua(u)\left|\partial_x\left(\frac{1}{\sqrt{u}}\partial_x\Sigma(u)\right)\right|^2 dx\geq\,0.
\end{align*}
\end{rem}
We remark that if $a(u)\equiv \text{const.}$, the above identity is exactly the time-evolution of the Fisher information for
the Boltzmann--Shannon entropy~\eqref{wstepne_bis}.
\vspace{2mm}

\begin{pr}{Theorem \ref{thm;1D-FI}}
The proof of \eqref{eqn;entropy} follows immediately by integration by parts.
Indeed,
\begin{align*}
\frac{d}{dt}\int_{0}^1H(u) dx
=\,&\int_{0}^1 \Lambda(u) \partial_t u dx
\\
=\,&-\int_{0}^1 a(u) \partial_x \Lambda(u) \partial_x u dx
\\
=\,&-\int_{0}^1 \frac{a(u)^2}{u} |\partial_x u|^2 dx
=\,-\int_{0}^1 |\partial_x \Sigma(u)|^2 dx.
\end{align*}
Noticing from the definition \eqref{eqn;plimitive} of $\Sigma$ that
\begin{align}
\partial_x\Sigma(u)
=\frac{a(u)}{\sqrt{u}}\partial_xu,
\quad
\partial_t\Sigma(u)=\frac{a(u)}{\sqrt{u}}\partial_tu,
\quad
\Sigma''(s)=\frac{a'(s)}{\sqrt{s}}-\frac{a(s)}{2s\sqrt{s}},
\label{eqn;sigma-formulation}
\end{align}
we then see
\begin{align}
\frac{1}{2}\frac{d}{dt}\int_0^1|\partial_x \Sigma(u)|^2dx
=\,&\int_0^1\partial_x\Sigma(u) \partial_x\partial_t\Sigma(u)dx
\nonumber\\
=\,&\int_0^1\partial_x\Sigma(u) \partial_x\left(\frac{a(u)}{\sqrt{u}} \partial_tu\right)dx
\nonumber\\
=\,&\int_0^1\partial_x\Sigma(u) \partial_x\left(\frac{a(u)}{\sqrt{u}} \partial_x\big[a(u)\partial_x u\big]\right)dx
\nonumber\\
=\,&\int_0^1\partial_x\Sigma(u) \partial_x\left(\frac{a(u)}{\sqrt{u}} \big[a'(u)(\partial_x u)^2+a(u)\partial_{xx}u\big]\right)dx.
\label{eqn;1D-1}
\end{align}
Since
\begin{align*}
&\frac{a(u)}{\sqrt{u}} \big[a'(u)(\partial_x u)^2+a(u)\partial_{xx}u\big]
\\
=\,&a(u) \left\{\left[\frac{a'(u)}{\sqrt{u}}-\frac{a(u)}{2u\sqrt{u}}\right](\partial_x u)^2+\frac{a(u)}{\sqrt{u}}\partial_{xx}u\right\}+\frac{a(u)^2}{2u\sqrt{u}}(\partial_xu)^2
\\
=\,&a(u) \left\{\Sigma''(u)(\partial_xu)^2+\Sigma'(u)\partial_{xx}u\right\}+\frac{1}{2\sqrt{u}} \left| \Sigma'(u) \partial_xu\right|^2
\\
=\,&a(u) \partial_{xx}\Sigma(u) +\frac{1}{2\sqrt{u}}\left|\partial_x\Sigma(u)\right|^2,
\end{align*}
it follows from the integral \eqref{eqn;1D-1} that
\begin{align}
\frac{1}{2}\frac{d}{dt}\int_0^1|\partial_x \Sigma(u)|^2dx
=\,&\int_0^1\partial_x\Sigma(u) \partial_x\left(a(u) \partial_{xx}\Sigma(u) +\frac{1}{2\sqrt{u}}\left|\partial_x\Sigma(u)\right|^2\right)dx
\nonumber\\
=\,&\int_0^1\partial_x\Sigma(u) \partial_x\Big(a(u) \partial_{xx}\Sigma(u) \Big)dx
\nonumber\\
&+\int_0^1\partial_x\Sigma(u) \partial_x\left(\frac{1}{2\sqrt{u}}\left|\partial_x\Sigma(u)\right|^2\right)dx.
\label{eqn;1D-2}
\end{align}
On the one hand, the first term on the right hand side of \eqref{eqn;1D-2} implies by integration by parts that
\begin{align*}
\int_0^1\partial_x\Sigma(u) \partial_x\Big(a(u) \partial_{xx}\Sigma(u) \Big)dx
=\,-\int_0^1 a(u) \left|\partial_{xx}\Sigma(u)\right|^2dx.
\end{align*}
On the other hand, the second term of \eqref{eqn;1D-2} is represented as
\begin{align*}
\int_0^1\partial_x\Sigma(u) \partial_x\left(\frac{1}{2\sqrt{u}}\left|\partial_x\Sigma(u)\right|^2\right)dx
=\,&-\int_0^1\frac{1}{4ua(u)} \left|\partial_x\Sigma(u)\right|^4 dx
\\
&+\int_0^1\frac{1}{\sqrt{u}} \left|\partial_x\Sigma(u)\right|^2 \partial_{xx}\Sigma(u) dx,
\end{align*}
where we used \eqref{eqn;sigma-formulation} to see that
\begin{align*}
\partial_x\left(\frac{1}{2\sqrt{u}}\left|\partial_x\Sigma(u)\right|^2\right)
=\,&-\frac{\partial_x u}{4u\sqrt{u}} \left|\partial_x\Sigma(u)\right|^2+\frac{1}{\sqrt{u}} \partial_x\Sigma(u) \partial_{xx}\Sigma(u)
\\
=\,&-\frac{1}{4u a(u)} \partial_x\Sigma(u) \left|\partial_x\Sigma(u)\right|^2+\frac{1}{\sqrt{u}} \partial_x\Sigma(u) \partial_{xx}\Sigma(u).
\end{align*}
Combining the above computations, we end up with
\begin{align*}
&\frac{1}{2}\frac{d}{dt}\int_0^1|\partial_x \Sigma(u)|^2dx
\\
=\,&-\int_0^1 a(u) \left|\partial_{xx}\Sigma(u)\right|^2dx+\int_0^1\frac{1}{\sqrt{u}} \left|\partial_x\Sigma(u)\right|^2 \partial_{xx}\Sigma(u) dx
\\
&-\int_0^1\frac{1}{4ua(u)} \left|\partial_x\Sigma(u)\right|^4 dx
\\
=\,&-\int_0^1\left|\sqrt{a(u)}\partial_{xx}\Sigma(u) - \frac{1}{2\sqrt{u a(u)}} \left|\partial_x\Sigma(u) \right|^2\right|^2 dx
\\
=\,&-\int_0^1ua(u)\left|\partial_x\left(\frac{1}{\sqrt{u}}\partial_x\Sigma(u)\right)\right|^2 dx,
\end{align*}
as desired.
\end{pr}
\vspace{3mm}
\begin{rem}
Following the proof of Theorem \ref{thm;1D-FI}, one notices that in higher dimensions we obtain the following identity.
For any $n\ge2$ we have
\begin{align*}
&\frac{1}{2}\frac{d}{dt}\int_{\Omega}|\nabla \Sigma(u)|^2dx
\\
=\,& - \int_{\Omega} a(u) | \Delta \Sigma(u)|^2dx+\int_{\Omega} \frac{1}{\sqrt{u}}(\nabla \Sigma (u))^{T} D^2\Sigma(u) \nabla \Sigma(u)dx
\\
&- \int_{\Omega} \frac{1}{4ua(u) } |\nabla \Sigma(u)|^4dx
\\
=\,&-\int_{\Omega} ua(u) \left| \nabla \left(\frac{1}{\sqrt{u}} \nabla \Sigma(u)\right) \right|^2 dx
\\
&-\int_{\Omega} a(u)|\Delta \Sigma(u)|^2dx + \int_{\Omega} a(u)|D^2 \Sigma(u)|^2dx.
\end{align*}
\end{rem}
\vspace{3mm}
\section{Related functional inequalities}\label{trzecia}

As we already mentioned, the Fisher information has been recently used in the studies of thermoelasticity (see \cite{Bi-Ci23, Bi-Ci25,Bi-Ci-La-Mu-Tr}) and combustion, see \cite{Li-Ya24}.
Its applicability to systems of PDEs is strongly related to the functional inequality \eqref{wstepne}.
In view of the latter inequality, the gradients of bad terms can be absorbed into the negative of the dissipation term of the Fisher information.
In case of nonlinear Fisher information, the role of the dissipation term is played by
\[
\int_{\Omega}ua(u)\left|\nabla \left(\frac{1}{\sqrt{u}}\nabla\Sigma(u)\right)\right|^2 dx.
\]
We shall hence give the corresponding functional inequalities relating dissipation terms of nonlinear Fisher information, in any space dimension, hoping that they can be used in systems of PDEs involving the nonlinear diffusion equation.
\begin{theorem}\label{thm;Fisher-ineq}
For $n\geq 1$, let $\Omega\subset\r^n$ be a smooth bounded domain.
For any positive function $u\in C^2(\bar{\Omega})$ satisfying $\nabla u \cdot \nu =0$ on $\partial\Omega$ and any smooth function $a:[0,\infty)\rightarrow \r_+$
\begin{align}
\int_{\Omega}  \frac{a(u)^3}{u^3}\left|\nabla u\right|^4 dx
\leq\,&(1+\sqrt{n})^2\int_{\Omega} ua(u)\left| \nabla\left( \frac{1}{\sqrt{u}} \nabla \Sigma(u)\right)  \right|^2dx.
\label{eqn;Beris}
\end{align}
Moreover, suppose that $a(s)$ is bounded from below, that is, there is a constant $\lambda>0$ such that
\[
a(s) \geq \lambda\quad \text{for all}~~s\geq 0.
\]
Then,
\begin{align}
\int_{\Omega} |D^2 \Sigma (u) |^2dx
\leq\,&\frac{4+(1+\sqrt{n})^2}{2\lambda}\int_{\Omega} u a(u)\left|\nabla \left(\frac{1}{\sqrt{u}}\nabla \Sigma(u)\right)\right|^2dx,
\label{eqn;Fisher-ineq}
\end{align}
where $\Sigma(s)$ is defined in \eqref{eqn;plimitive} as
\begin{align*}
\Sigma(s):=\int_1^s \frac{a(\tau)}{\sqrt{\tau}} d\tau.
\end{align*}
\end{theorem}
We point out from the definition of $\Lambda(u)$ defined in \eqref{eqn;premitive-higher} that
\begin{align*}
 \int_{\Omega} u a(u)\left|\nabla \left(\frac{1}{\sqrt{u}}\nabla \Sigma(u)\right)\right|^2dx
 =\,& \int_{\Omega} u a(u)\left|D^2\Lambda(u)\right|^2dx
 \\
 \le\,& n\int_{\Omega} u a(u)\left|\Delta\Lambda(u)\right|^2dx.
\end{align*}
It is worthwhile to state that in the linear diffusion case $a(s)\equiv\text{const.}$, the inequality~\eqref{eqn;Fisher-ineq} established in Theorem~\ref{thm;Fisher-ineq} is exactly \eqref{wstepne} obtained in~\cite{Ci-Ma-Ka-Mi}, see also the recent Riemannian setting extension in~\cite{Ci-Ga-Kr}.
\begin{align*}
\int_{\Omega} \frac{|\nabla u|^4}{u^3}\,dx \le\,&C \int_{\Omega} u |D^2\log u |^2dx.
\end{align*}
The inequality \eqref{eqn;Beris} is a generalization of the Bernis type inequality, see for instance \cite{Wi12}.
We point out that the version in \cite{Wi12}, also nonlinear is different from the one we present. We also note that its versions have already turned out to be applicable in chemotaxis system as well as chemotaxis--Navier--Stokes system, see \cite{Wi12}. We also mention the recent Riemannian geometry version in \cite{Ci-Ga-Kr}.
\vspace{2mm}

\begin{pr}{Theorem \ref{thm;Fisher-ineq}}
We begin with decomposing the integrand appearing on the left hand side of \eqref{eqn;Fisher-ineq}.
Under the definition of $\Sigma(u)$ and assumption on $a(u)$,  we see
\begin{align*}
&\int_{\Omega} \left|  D^2 \Sigma(u)\right|^2dx
\\
=\,&\int_{\Omega}\sum_{i,j=1}^{n} \left| \partial_i \left(\frac{\sqrt{u}}{\sqrt{u}} \partial_j \Sigma(u)\right)  \right|^2 dx
\\
=\,&\int_{\Omega}\sum_{i,j=1}^n \left| \sqrt{u}  \partial_i \left( \frac{1}{\sqrt{u}}\partial_j \Sigma(u)\right) + \frac{a(u)}{2u \sqrt{u}}\partial_i u \partial_j u\right|^2  dx
\\
=\,&\int_{\Omega} \frac{1}{a(u)}\sum_{i,j=1}^n\left| \sqrt{ua(u)}  \partial_i \left( \frac{1}{\sqrt{u}}\partial_j \Sigma(u)\right) + \frac{a(u)\sqrt{a(u)}}{2u \sqrt{u}}\partial_i u \partial_j u\right|^2dx
\\
\leq\,&\frac{1}{\lambda}\int_{\Omega}\sum_{i,j=1}^n\left| \sqrt{ua(u)}  \partial_i \left( \frac{1}{\sqrt{u}}\partial_j \Sigma(u)\right) + \frac{a(u)\sqrt{a(u)}}{2u \sqrt{u}}\partial_i u \partial_j u\right|^2 dx
\\
\le\,&\frac{2}{\lambda}\int_{\Omega} \sum_{i,j=1}^n\left[ \left| \sqrt{ua(u)}  \partial_i \left( \frac{1}{\sqrt{u}}\partial_j \Sigma(u)\right) \right|^2+ \left|\frac{a(u)\sqrt{a(u)}}{2u \sqrt{u}}\partial_i u \partial_j u\right|^2 ~\right] dx.
\end{align*}
Hence,
\begin{align}
&\int_{\Omega} \left|  D^2 \Sigma(u)\right|^2\,dx \notag
\\
\leq\,&\frac2\lambda\int_{\Omega}\left| \sqrt{ua(u)}  \nabla \left( \frac{1}{\sqrt{u}}\nabla \Sigma(u)\right) \right|^2 \,dx
+\frac{1}{2\lambda}\int_{\Omega}  \frac{a(u)^3}{u^3}\left|\nabla u\right|^4 \,dx.
\label{eqn;sigma-lambda-ineq}
\end{align}
In order to estimate the second term on the right hand side of \eqref{eqn;sigma-lambda-ineq}, let us recall the primitive functional~$\Lambda(s)$ defined in \eqref{eqn;premitive-higher}.
Integrating by parts the last term on the right-hand side of \eqref{eqn;sigma-lambda-ineq}, we obtain
\begin{align}
&\frac{1}{2\lambda}\int_{\Omega}  \frac{a(u)^3}{u^3}\left|\nabla u\right|^4 dx
\nonumber\\
=\,&\frac{1}{2\lambda}\int_{\Omega} |\nabla \Lambda(u)|^2 \nabla \Lambda(u)\cdot \nabla u dx
\nonumber\\
=\,&-\frac{1}{2\lambda}\int_{\Omega} u |\nabla \Lambda(u)|^2\Delta \Lambda(u)dx-\frac{1}{2\lambda}\int_{\Omega}u \nabla |\nabla \Lambda(u)|^2\cdot \nabla \Lambda(u)dx
\nonumber\\
=\,&-\frac{1}{2\lambda}\int_{\Omega} u |\nabla \Lambda(u)|^2\Delta \Lambda(u)dx-\frac{1}{\lambda}\int_{\Omega}u \left(\nabla\Lambda(u) \right)^T D^2 \Lambda(u) \nabla \Lambda(u)dx,
\label{eqn;lambda-ineq}
\end{align}
where
\[
\nabla \Lambda(u)\cdot \nu=\,\frac{a(u)}{u}\nabla u \cdot \nu =0\quad\text{on}~~\partial\Omega.
\]
Let us first deal with the second term on the right hand side of \eqref{eqn;lambda-ineq}.
Applying the Cauchy--Schwarz inequality, we obtain
\begin{align*}
&\left| \int_{\Omega}u \left(\nabla\Lambda(u) \right)^T D^2 \Lambda(u) \nabla \Lambda(u)dx   \right|
\\
\leq\,&\left(\int_{\Omega} \left| \sqrt{ua(u)} D^2\Lambda(u)  \right|^2dx \right)^{\frac{1}{2}}
\left( \int_{\Omega} \left| \frac{u}{\sqrt{u a(u)}} |\nabla \Lambda(u)|^2  \right|^2 dx \right)^{\frac{1}{2}}
\\
=\,&\left(\int_{\Omega} \left| \sqrt{ua(u)} \nabla\left( \frac{1}{\sqrt{u}} \nabla \Sigma(u)\right)  \right|^2dx \right)^{\frac{1}{2}}
\left( \int_{\Omega} \frac{a(u)^3}{u^3} |\nabla u|^4dx\right)^{\frac{1}{2}}.
\end{align*}
In a similar way, the first term on the right hand side of \eqref{eqn;lambda-ineq} is estimated as
\begin{align*}
&\left| \int_{\Omega} u |\nabla \Lambda(u)|^2\Delta \Lambda(u)dx \right|
\\
\leq\,&\int_{\Omega} u |\nabla \Lambda(u)|^2 |\Delta \Lambda(u)|dx
\\
\leq\,&\sqrt{n}\int_{\Omega}u |\nabla \Lambda(u)|^2 |D^2 \Lambda(u)|dx
\\
\leq\,&\sqrt{n}\left(\int_{\Omega} \left| \sqrt{ua(u)} \nabla\left( \frac{1}{\sqrt{u}} \nabla \Sigma(u)\right)  \right|^2dx \right)^{\frac{1}{2}}
\left( \int_{\Omega} \frac{a(u)^3}{u^3} |\nabla u|^4dx\right)^{\frac{1}{2}},
\end{align*}
where we use the point-wise estimate $|\Delta u | \leq \sqrt{n} |D^2 u |$ and the Cauchy--Schwarz inequality.
Hence, collecting the above estimates leads us to
\begin{align*}
&\frac{1}{2\lambda}\int_{\Omega}  \frac{a(u)^3}{u^3}\left|\nabla u\right|^4 dx
\\
\leq\,&\frac{1+\sqrt{n}}{2\lambda}\left(\int_{\Omega} \left| \sqrt{ua(u)} \nabla\left( \frac{1}{\sqrt{u}} \nabla \Sigma(u)\right)  \right|^2dx \right)^{\frac{1}{2}}
\left( \int_{\Omega} \frac{a(u)^3}{u^3} |\nabla u|^4dx\right)^{\frac{1}{2}},
\end{align*}
and by the positivity of $u$ and $a(u)$
\begin{align*}
\int_{\Omega}  \frac{a(u)^3}{u^3}\left|\nabla u\right|^4 dx
\leq\,&(1+\sqrt{n})^2\int_{\Omega} ua(u)\left| \nabla\left( \frac{1}{\sqrt{u}} \nabla \Sigma(u)\right)  \right|^2dx.
\end{align*}
Consequently, we have from \eqref{eqn;sigma-lambda-ineq},
\begin{align*}
\int_{\Omega} \left|  D^2 \Sigma(u)\right|^2dx
\leq\,&\frac{4+(1+\sqrt{n})^2}{2\lambda}\int_{\Omega}ua(u)\left|   \nabla \left( \frac{1}{\sqrt{u}}\nabla \Sigma(u)\right) \right|^2  dx,
\end{align*}
which ends the proof.
\end{pr}

\section{Applications to the fully parabolic quasilinear critical 1d Keller--Segel system}\label{czwarta}

Consider the following one-dimensional quasilinear Keller--Segel system with nonlinear sensitivity:
\begin{align}\label{P1}
	\begin{cases}
	u_t =\partial_x \left(D (u)\partial_x u - S (u)\partial_x v \right)
	&\mathrm{in}\ (0,T)\times(0,1), \\[1mm]
	v_t = \partial_{xx} v  -v+u
	&\mathrm{in}\ (0,T)\times(0,1),  \\[1mm]
	 \partial_x u = \partial_x v = 0 &\mathrm{on} \ (0,T) \times \{0, 1\}, \\[1mm]
	u(0,\cdot)=u_0,\quad v(0,\cdot)=v_0
	&\mathrm{in}\ (0,1),
	\end{cases}
\end{align}
where
$$
D(u) = (1+u)^{-p},\quad S(u) = u(1+u)^{-q} \qquad \mbox{with }p,q\in \r.
$$
Furthermore we assume $(u_0, v_0) \in (W^{1,\infty}(0,1))^2$ is the pair of nonnegative initial data.
Local existence and uniqueness of classical positive solutions are known, see \cite{BBTW}.
Let us first recall the mass conservation law:
$$
\|u(t)\|_{L^1(0,1)} = \|u_0\|_{L^1(0,1)}=:M, \qquad t\in(0,T),
$$
and the classical Lyapunov functional associated with \eqref{P1}, see \cite{BBTW, Ci-Fu18, Fu};
the following identity holds:
\begin{eqnarray}\label{Lyap_KS}
\frac{d}{dt}\mathcal{L}(u, v)
+
\int_0^1 | v_t|^2dx + \int_0^1 S(u) \cdot \left|\frac{D (u)}{S(u)}\partial_xu -\partial_xv \right|^2dx=0,
\end{eqnarray}
where
\begin{eqnarray*}
\mathcal{L}(u, v) &:=& \int_0^1 G(u)dx - \int_0^1 uv dx+  \frac{1}{2} \|v\|^2_{H^1(0,1)},\\
G(s) &:=& \int_{1}^{s} \int_{1}^{\sigma} \frac{D (\tau)}{S(\tau)}d\tau d\sigma.
\end{eqnarray*}

\begin{rem}\label{pr_critical}
In the earlier literature (see the survey \cite[Section 3.6]{BBTW}),
$p-q = 1$ is identified to be the critical exponent.
However, the question whether the mass critical phenomenon appears or not in the full generality of $(p,q)$ is open.
The special case $(p,q)=(1,0)$ was studied and the absence of the mass critical phenomenon was proved in \cite{Ci-Fu18} by introducing a new functional, see Subsection \ref{4.5}.
Moreover,  the case $(p,q)=(2,1)$ will be considered in Subsection \ref{4.66} and the absence of the mass critical phenomenon is newly established in Theorem \ref{MT}. The cases of other $(p,q)$ are still open.
\end{rem}
\begin{rem}\label{poor}
In the higher dimensional case, the mass critical phenomenon was established by using the above Lyapunov functional $\mathcal{L}$, see~\cite[Section 3.6]{BBTW}.
On the other hand, the Lyapunov functional $\mathcal{L}$ gives very poor information in the one dimensional setting.
Actually, the one dimensional Sobolev inequality and the mass conservation law imply
$$
\int_0^1 uv dx\leq M\|v\|_{L^\infty(0,1)} \leq CM\|v\|_{H^1(0,1)}\leq \frac{1}{2} \|v\|^2_{H^1(0,1)} + \frac{C^2M^2}{2},
$$
and then the upper and lower bound for the Lyapunov functional are derived:
\begin{eqnarray}\label{Lyap_est_KS}
\sup_{t>0}|\mathcal{L}(u(t), v(t))| < \infty.
\end{eqnarray}
On the other hand, the condition $p-q = 1$ implies $\frac{D(u)}{S(u)} =\frac{1}{u(1+u)}$ and then $G(u) \simeq \log u $.
Hence, the bound of the Lyapunov functional $\mathcal{L}$ implies a priori estimate on $\|\log u\|_{L^1(0,1)}$, which is a weaker information than the one coming from the mass conservation law.
\end{rem}

\subsection{Alternative derivation of the functional in \cite{Ci-Fu18}}\label{4.5}

In this section we consider a special case \eqref{P1}, namely with $(p,q)=(1,0)$.
This case is already well-known, the global existence result was established in \cite{Ci-Fu18}.
The crucial new functional found there led to new estimate yielding global existence of solutions (their boundedness was established in \cite{BCFS}).
We remark that this functional in \cite{Ci-Fu18} corresponds to the Fisher information (entropy production) of the entropy identity \eqref{Lyap_KS}.
In this section we present a new proof of a crucial \cite[Lemma 3.1]{Ci-Fu18}.
Our proof is shorter and more direct, in particular the identity in \cite[Lemma 2.1]{Ci-Fu18} is now redundant.
It strongly relies on our approach to the nonlinear Fisher information in Section \ref{druga}.

Actually the following lemma could be stated for any $p\in \r$.
\begin{lem}\label{lem;iden}
Let $(u,v)$ solve \eqref{P1} with $S(u)=u$. Then,
\begin{align}
\label{eqn;KS-F-identity}
&\frac{d}{dt} \left( \frac12 \int_0^1 \frac{(D(u))^2}{u} |\partial_x u|^2dx\right)
+\int_0^1 uD(u) \left| \partial_x \left( \frac{D(u)}{u} \partial_x u\right) \right| ^2dx
\\
&=\int_0^1 uD(u) \partial_{xx} v \partial_x\left( \frac{D(u)}{u} \partial_x u\right)dx.
\notag
\end{align}
\end{lem}

\begin{rem}
Knowing Lemma \ref{lem;iden}, applying then \cite[Lemma 3.2 and Lemma 3.3]{Ci-Fu18}, we obtain

\begin{align*}
\frac{d}{dt} \mathcal{F}(u(t)) + \mathcal{G}(u(t),v(t)) =\,\int_0^1 \frac{u D(u) (v+v_t)^2}{4}dx,
\end{align*}
where
\begin{align*}
\mathcal{F}(u(t)) :=\,\frac12 \int_0^1 \frac{(D(u))^2}{u} |\partial_x u |^2dx - \int_{0}^1 u \int_1^u D(s) ds dx,
\end{align*}
and
\begin{align*}
\mathcal{G}(u(t),v(t)) :=\, \int_0^1 uD(u) \left|  \partial_x\left( \frac{D(u)}{u} \partial_xu \right) -\partial_{xx}v +\frac{v+v_t}{2}  \right|^2dx,
\end{align*}
which is a crucial identity to arrive at a global solution to \eqref{P1} with $(p,q)=(1,0)$.
\end{rem}

\begin{pr}{Lemma \ref{lem;iden}}
Calculations in Theorem \ref{thm;1D-FI} are very useful here.
Not only the formulation of Theorem \ref{thm;1D-FI}, which could be taken from the 1D version of Theorem \ref{thm;FI}.
Defining
\begin{align*}
\Sigma(u) = \int_1^u \frac{D(\tau)}{\tau}d\tau
\end{align*}
like in \eqref{eqn;plimitive}, we notice that \eqref{eqn;KS-F-identity} reads as
\begin{align*}\eqntag
\label{eqn;FI-KS}
&\frac12 \frac{d}{dt} \int_0^1 |\partial_x\Sigma(u)|^2 dx
\\
=\,& - \int_0^1 uD(u) \left| \partial_x \left(\frac{1}{\sqrt{u}} \partial_x \Sigma(u) \right)\right|^2 dx+ \int_0^1 uD(u) \partial_{xx} v \partial_x\left(\frac{D(u)}{u} \partial_x u\right)dx.
\end{align*}
Let us multiply the upper equation in \eqref{P1} with $S(u)=u$ by $-\frac{D(u)}{\sqrt{u}} \partial_{xx}\Sigma(u)$
to see that the proof of \eqref{eqn;KS-F-identity} reduces to showing that
\begin{align}
\int_0^1  \partial_x(u\partial_xv) \frac{D(u)}{\sqrt{u}} \partial_{xx} \Sigma(u)dx
=\int_0^1 \partial_x\left( \frac{D(u)}{u} \partial_xu\right) uD(u) \partial_{xx}vdx.
\label{eqn;KS-non}
\end{align}
Indeed, the first term on right hand side of \eqref{eqn;FI-KS} is obtained exactly same way as in Theorem \ref{thm;1D-FI}.
 In order to prove \eqref{eqn;KS-non}, we calculate by the definition of $\Sigma$ that
\begin{align}
&\int_0^1  \partial_x(u\partial_xv) \frac{D(u)}{\sqrt{u}}  \partial_{xx}\Sigma(u)dx\notag
\\
=\,& \int_0^1 \partial_x(u \partial_xv) D(u) \left[ \frac{1}{\sqrt{u}} \partial_x\left(  \frac{D(u)}{\sqrt{u}}\partial_xu\right)  \right]dx\notag
\\
=\,&\int_0^1 \partial_x(u\partial_xv) D(u)  \left[\partial_x \left(\frac{D(u)}{u}\partial_xu\right) + \frac{D(u)(\partial_xu)^2}{2 u^2} \right]dx \notag
\\
=\,&\int_0^1 \left( \partial_xu \partial_xv D(u) + uD(u) \partial_{xx} v \right)\left[ \partial_x\left(\frac{D(u)}{u}\partial_xu\right) + \frac{D(u)(\partial_xu)^2}{2 u^2} \right]dx\notag
\\
=\,&\int_0^1\partial_x \left( \frac{D(u)}{u} \partial_xu\right) uD(u) \partial_{xx}vdx
+\int_0^1 \partial_x\left(\frac{D(u)}{u}\partial_xu\right) D(u) \partial_xu\partial_x vdx \notag
\\
&+\frac12 \int_0^1 \frac{(\partial_xu)^2 D(u)^2}{u^2}\partial_xu \partial_xvdx + \frac12 \int _0^1\frac{(\partial_xu)^2 D(u)^2}{u} \partial_{xx}vdx.
\label{eqn:ltFI}
\end{align}
Here, the integration by parts implies
\begin{align*}
&\int_0^1 \partial_x\left(\frac{D(u)}{u}\partial_xu\right) D(u)\partial_x u\partial_x v dx
\\
=\,& \frac12 \int_0^1 \partial_x\left[ \left(\frac{D(u)}{u} \partial_xu\right)^2  ~\right] u\partial_x vdx
\\
=\,&-\frac12 \int_0^1 \frac{(\partial_xu)^2 D(u)^2}{u^2}\partial_xu \partial_xvdx
 - \frac12 \int_0^1 \frac{(\partial_xu)^2 D(u)^2}{u} \partial_{xx}vdx.
\end{align*}
Plugging the above in \eqref{eqn:ltFI}, we arrive at \eqref{eqn;KS-non}.
The proof is finished.
\end{pr}
\vspace{2mm}

\subsection{Global solutions in the 1D fully parabolic Keller--Segel system with nonlinear diffusion and sensitivity}\label{4.66}

In the present section, we apply our approach to obtain a straightforward proof of the identity
in~\cite[Proposition 2.4]{Fu}.
Next. we utilize the latter and obtain the global existence of a unique solution to \eqref{P1} for the choice $(p,q)=(2,1)$,
which leads to the critical exponent.
Such a result has not been known so far, see Remark \ref{pr_critical}.

Our main theorem in this section reads.
\begin{theorem}\label{MT}
Assume that $(p,q) = (2,1)$.
Then the problem \eqref{P1} has a unique classical positive solution, which exists globally in time.
\end{theorem}
The following proposition is crucial in the proof of Theorem \ref{MT}.
\begin{prop}\label{prop_regularity_ep}
Assume that $p-q = 1$ and $q\in \left(\frac{1}{2}, 1\right]$.
Let $(u,v)$ be the classical positive solution of \eqref{P1} in $(0,T)\times (0,1)$.
Then, there exists some $C=C(T)>0$ such that for all $t \in (0,T)$,
$$
\int_0^1 \frac{|\partial_xu|^2}{u(1+u)^{p+1}}dx \leq C.
$$
\end{prop}

We recall the following entropy production identity in \cite[Proposition 2.4]{Fu}.
\begin{prop}\label{prop_new_lyapunov}
Let $(u,v)$ be a solution of \eqref{P1} in $(0,T)\times (0,1)$.
The following identity is satisfied:
\begin{align}\label{entro_prod}
\nonumber
&\frac{d}{dt} \mathcal{F}(u(t)) + \mathcal{D}(u(t),v(t))
\\
=\,&\int_0^1 \frac{S(u)D(u)(v+ v_t)^2}{4}dx\nonumber
\\&+ \int_0^1 \left(\frac{D(u)}{S(u)}\partial_xu_x - \partial_xv\right) \cdot
\frac{(D(u))^2 S''(u)}{2S(u)} |\partial_xu|^2\partial_xudx,
\end{align}
where
\begin{eqnarray*}
\mathcal{F}(u(t))
&:=&\frac{1}{2} \int_0^1 \frac{(D(u))^2}{S(u)}|\partial_xu|^2dx - \int_0^1 \Psi(u)dx, \\[2mm]
\Psi(\phi) &:=& \int_1^\phi \left( \int_1^r \frac{\tau D(\tau) S'(\tau)}{S(\tau)}d\tau  + r D(r) \right)dr,\\[2mm]
\mathcal{D}(u(t),v(t))
&:=& \int_0^1 S(u) D(u) \left|\partial_x \left(\frac{D(u)}{S(u)}\partial_xu \right) -\partial_{xx}v  + \frac{(v+v_t)}{2} \right|^2dx.
\end{eqnarray*}
\end{prop}
%
%

\begin{rem}
The second term of right hand side in \eqref{entro_prod} disturbs us to apply similar estimates to the ones in \cite{Ci-Fu18, BCFS}.
\end{rem}

Let us first apply our approach to give a simpler and more direct proof of~\eqref{entro_prod}. It is again based on the calculations in the proof of Theorem \ref{thm;1D-FI}.
\vspace{2mm}

\begin{pr}{Proposition \ref{prop_new_lyapunov}}
The proof of Proposition \ref{prop_new_lyapunov} reduces to proving the following identity
\begin{align}
\label{eqn;identity-NSD}
&\frac{d}{dt}\left(\frac12 \int_0^1 \frac{D(u)^2}{S(u)}|\partial_xu|^2 dx\right)+\int_0^1 D(u)S(u) \left| \partial_x \left(\frac{D(u)}{S(u)}\partial_x u\right) \right|^2dx
\\
=\,&\int_0^1 D(u) S(u) \partial_{xx}v \partial_x \left(\frac{D(u)}{S(u) }\partial_xu \right)dx\notag
\\
&+\int_0^1 \left( \frac{D(u)}{S(u)} \partial_x u - \partial_x v  \right) \frac{D(u)^2S^{''}(u)}{2S(u)} |\partial_xu|^2 \partial_xudx.\notag
\end{align}
Indeed, knowing \eqref{eqn;identity-NSD}, we are in a position of \cite[Lemma 2.2]{Fu}.
Then, we repeat the argument in \cite{Fu}. In order to prove \eqref{eqn;identity-NSD}, we rewrite the upper equation \eqref{P1} as follows:
\begin{equation}
 u_t = \partial_x \left( D(u) \partial_x u - S(u) \partial_x v \right)
=\dx \left(S(u) \cdot \dx \left( \int_1^u \frac{D(\tau)}{S(\tau)}\,d\tau - v\right)\right).
\label{P2}
\end{equation}
In what follows we will denote
\begin{equation*}
\Sigma(s):=\int_1^s \frac{D(\tau)}{\sqrt{S(\tau)}}d\tau.
\end{equation*}
Multiplying \eqref{P2} by $-\frac{D(u)}{\sqrt{S(u)}} \partial_{xx}\Sigma(u)$, and integrating by parts, we obtain
\begin{align*}
\frac12 \frac{d}{dt} \int_0^1 |\partial_x\Sigma(u)|^2dx
=\,& \int_0^1 \partial_x \Sigma (u)  \partial_x\partial_t \Sigma (u) dx\\
=\,& -\int_0^1 \partial_{xx}\Sigma (u)  \partial_t \Sigma (u) dx\\
=\,& - \int_{0}^1 u_t \frac{D(u)}{\sqrt{S(u)}}\partial_{xx} \Sigma(u)dx\\
=\,&- \int_{0}^1\partial_x (D(u)\partial_xu) \frac{D(u)}{\sqrt{S(u)}}\partial_{xx} \Sigma(u)dx
\\
&+\int_0^1 \partial_x(S(u)\partial_xv) \frac{D(u)}{\sqrt{S(u)}} \partial_{xx}\Sigma(u)dx
=:\,I+II.
\eqntag\label{eqn;I-II}
\end{align*}
Let us first proceed with $II$.
By the similar way as in \eqref{eqn:ltFI}, we have
\begin{align*}
II=\,&\int_0^1 \left( S(u)\partial_{xx}v + S'(u) \partial_xu\partial_x v \right) \frac{D(u)}{\sqrt{S(u)}} \partial_x\left(\frac{D(u)}{\sqrt{S(u)}} \partial_xu\right)dx
\\
=\,&\int_0^1\left( S(u)\partial_{xx}v+ S'(u) \partial_xu\partial_x v \right) D(u) \left[  \frac{1}{\sqrt{S(u)}    }  \partial_x\left(\frac{D(u)}{\sqrt{S(u)}} \partial_xu\right) \right]dx
\\
=\,&\int_0^1 \left( S(u)\partial_{xx}v + S'(u) \partial_xu\partial_x v \right) D(u) \partial_x\left(  \frac{D(u)}{S(u)} \partial_xu\right)dx
\\& + \int_0^1 \left( S(u)\partial_{xx}v + S'(u) \partial_xu\partial_x v \right) D(u) \frac{S'(u) (\partial_xu)^2 D(u)}{2S^2(u)} dx
\\
=\,&
\int_0^1 D(u)S(u) \partial_{xx}v \partial_x\left(  \frac{D(u)}{S(u)} \partial_xu \right)dx
\\&+ \int_0^1 D(u) S'(u) \partial_xu\partial_x v  \partial_x\left(  \frac{D(u)}{S(u)} \partial_xu \right)dx
\\
&+\int_0^1 \frac{S'(u) D(u) (\partial_xu)^2}{2S^2(u)} D(u)S(u) \partial_{xx}vdx
\\&+\int_0^1 \frac{[S'(u)]^2 (\partial_xu)^2 D(u)^2}{2S^2(u)}\partial_x u \partial_xvdx.
\eqntag \label{eqn;KS-NDS}
\end{align*}
Here, the integration by parts implies
\begin{align*}
 &\int_0^1 D(u) S'(u) \partial_xu\partial_x v  \partial_x\left(  \frac{D(u)}{S(u)} \partial_xu \right)dx
 \\
 =\,& \frac12\int_0^1 \partial_x\left( \left[ \frac{D(u)}{S(u)} \partial_xu \right]^2 \right)  S(u) S'(u) \partial_xvdx
 \\
 =\,& - \frac12 \int_0^1 \frac{D(u)^2(\partial_xu)^2}{S^2(u)} S(u) S'(u) \partial_{xx}vdx
 \\& -\frac12 \int_0^1 \frac{D(u)^2(\partial_xu)^2}{S(u)^2}S''(u) S(u)\partial_x v\partial_xudx
 \\& -\frac12 \int_0^1 \frac{D(u)^2(\partial_xu)^2}{S^2(u)} [S'(u)]^2 \partial_xv \partial_xudx.
\end{align*}
Plugging the above in \eqref{eqn;KS-NDS}, we arrive at
\begin{align*}
II=\,&\int D(u)S(u) \partial_{xx}v \partial_x\left(  \frac{D(u)}{S(u)} \partial_xu \right)dx
\\&-\frac12 \int_0^1 \frac{D(u)^2(\partial_xu)^2}{S(u)}S''(u)  \partial_xv\partial_xudx.
\eqntag\label{eqn;II}
\end{align*}
Let us now estimate $I$ based on the same spirit in the proof of Theorem \ref{thm;1D-FI}. By the integration by parts, it follows
\begin{align*}
I=\,&\int_0^1\dx \Sigma(u) \dx\left(\frac{D(u)}{\sqrt{S(u)}} \dx\big[D(u)\dx u\big]\right) dx
\\
=\,&\int_0^1\dx  \Sigma(u)  \dx\left(\frac{D(u)}{\sqrt{S(u)}} \big[D'(u)(\dx u)^2+D(u)\partial_{xx}u\big]\right) dx.
\end{align*}
Since the definition of $\Sigma$ implies
\begin{align*}
&\frac{D(u)}{\sqrt{S(u)}} \big[D'(u)(\dx u)^2+D(u)\partial_{xx}u\big]
\\
=\,&
D(u) \left\{\left[\frac{D'(u)}{\sqrt{S(u)}}-\frac{D(u)S^\prime(u)}{2S(u)\sqrt{S(u)}}\right](\dx u)^2+\frac{D(u)}{\sqrt{S(u)}}\partial_{xx}u\right\}
\\&+\frac{D(u)^2S^\prime(u)}{2S(u)\sqrt{S(u)}}(\dx u)^2
\\
=\,&D(u) \left\{\Sigma''(u)(\dx u)^2+\Sigma'(u)\partial_{xx}u\right\}+\frac{S^\prime(u)}{2\sqrt{S(u)}} \left| \Sigma'(u) \dx u\right|^2
\\
=\,&D(u) \partial_{xx}\Sigma(u) +\frac{S^\prime(u)}{2\sqrt{S(u)}}\left|\dx\Sigma(u) \right|^2,
\end{align*}
we see that
\begin{align*}\label{eqn;1D-4}
\eqntag
I
=\,&
\int_0^1\dx\left( \Sigma(u) \right) \dx\left(D(u) \partial_{xx}\Sigma(u) +\frac{S^\prime(u)}{2\sqrt{S(u)}}\left|\dx \left(\Sigma(u) \right) \right|^2\right)dx
\\
\nn
=\,&
\int_0^1\dx\left( \Sigma(u) \right) \dx\Big(D(u) \partial_{xx}\left( \Sigma(u) \right) \Big)dx
\\
&+\int_0^1\dx \Sigma(u)  \dx\left(\frac{S^\prime(u)}{2\sqrt{S(u)}}\left|\dx\left(\Sigma(u)\right)\right|^2\right)dx.
\end{align*}
The first term on the right hand side of \eqref{eqn;1D-4} is
\begin{align}\label{F}
\int_0^1\dx \Sigma(u)  \dx\Big(D(u) \partial_{xx}\Sigma(u) \Big)dx
=-\int_0^1 D(u) \left|\partial_{xx}  \Sigma(u) \right|^2dx.
\end{align}
We next deal with the the second term on the right hand side of \eqref{eqn;1D-4}.
Since the definition of $\Sigma$ implies
\begin{align*}
&\dx\left(\frac{S^\prime(u)}{2\sqrt{S(u)}}\left|\dx\Sigma(u) \right|^2\right)\\
=\,&
\left(\frac{S^{\prime \prime}(u)}{2\sqrt{S(u)}} - \frac{(S^\prime(u))^2}{4S(u)\sqrt{S(u)}} \right) \dx u \left|\dx\Sigma(u) \right|^2
+\frac{S^\prime(u)}{\sqrt{S(u)}} \dx\Sigma(u)  \partial_{xx} \Sigma(u)
\\
=\,&
\left(\frac{S^{\prime \prime}(u)}{2\sqrt{S(u)}} - \frac{(S^\prime(u))^2}{4S(u)\sqrt{S(u)}} \right) \cdot \frac{\sqrt{S(u)}}{D(u)} \dx \Sigma(u)\cdot  \left|\dx\Sigma(u) \right|^2
\\&
+\frac{S^\prime(u)}{\sqrt{S(u)}} \dx \Sigma(u)  \partial_{xx} \Sigma(u)
\\
=\,&
\frac{2S(u)S^{\prime\prime}(u)-(S^\prime(u))^2}{4D(u)S(u)}
\left(\dx\Sigma(u) \right)^3
+\frac{S^\prime(u)}{\sqrt{S(u)}} \dx\Sigma(u)  \partial_{xx} \Sigma(u),
\end{align*}
the second term of \eqref{eqn;1D-4} can be represented as
\begin{align*}
&\int_0^1\dx \Sigma(u) \dx\left(\frac{S^\prime(u)}{2\sqrt{S(u)}}\left|\dx\Sigma(u)\right|^2\right)dx
\\=\,&
\io \frac{2S(u)S^{\prime\prime}(u)-(S^\prime(u))^2}{4D(u)S(u)}
\left|\dx\Sigma(u)\right|^4dx
+
\io \frac{S^\prime(u)}{\sqrt{S(u)}} \left|\dx \Sigma(u) \right|^2 \partial_{xx} \Sigma(u)dx.
\eqntag \label{eqn;KS-1d}
\end{align*}
Hence, \eqref{eqn;1D-4}, \eqref{F} and \eqref{eqn;KS-1d} yield
\begin{align*}
\eqntag
\label{entropypr_1}
I
=\,&
-\int_0^1 D(u) \left|\partial_{xx}  \Sigma(u) \right|^2dx
\nn
+
\io \frac{2S(u)S^{\prime\prime}(u)-(S^\prime(u))^2}{4D(u)S(u)}
\left|\dx\Sigma(u)\right|^4dx
\\&+
\io \frac{S^\prime(u)}{\sqrt{S(u)}} \left|\dx \Sigma(u) \right|^2 \partial_{xx} \Sigma(u)dx
\\
=\,&
-\int_0^1\left|\sqrt{D(u)}\partial_{xx}\Sigma(u)  - \frac{S^\prime(u)}{2\sqrt{D(u)S(u)}} \left|\dx \Sigma(u) \right|^2\right|^2 dx
\\&+
\frac{1}{2}\io \frac{S^{\prime\prime}(u)}{D(u)}
\left|\dx\Sigma(u) \right|^4dx.
\end{align*}
Simple calculations show that
\begin{align*}
&\sqrt{D(u)}\partial_{xx} \Sigma(u)  - \frac{S^\prime(u)}{2\sqrt{D(u)S(u)}} \left|\dx \Sigma(u) \right|^2\\
=\,&
\sqrt{D(u)}\partial_{xx} \Sigma(u) - \frac{S^\prime(u)}{2\sqrt{D(u)S(u)}}\cdot \frac{D(u)}{\sqrt{S(u)}}\dx u \cdot \dx \Sigma(u)\\
=\,& \sqrt{D(u)S(u)}\left[\frac{1}{\sqrt{S(u)}}\partial_{xx} \Sigma(u) - \frac{S^\prime(u)}{2S(u)\sqrt{S(u)}}\dx u \cdot \dx \Sigma(u) \right] \\
=\,&
\sqrt{D(u)S(u)} \dx \left( \frac{1}{\sqrt{S(u)}} \dx \Sigma(u) \right).
\end{align*}
Hence, owing to \eqref{entropypr_1}
\begin{align*}
I
=\,&
-\int_0^1\left| \sqrt{D(u)S(u)} \dx \left( \frac{1}{\sqrt{S(u)}} \dx \Sigma(u) \right) \right|^2 dx
\\&+
\frac{1}{2}\io \frac{S^{\prime\prime}(u)}{D(u)}
\left|\dx\Sigma(u)\right|^4dx.
\eqntag\label{eqn;I}
\end{align*}
Now, plugging \eqref{eqn;I} and \eqref{eqn;II} into \eqref{eqn;I-II}, we obtain \eqref{eqn;identity-NSD}.
\end{pr}

\vspace{3mm}
Let us now proceed towards the proof of Proposition \ref{prop_regularity_ep}.
The Idea of proof is the following: The second term of right hand side in \eqref{entro_prod} can be represented as
$$
 \int_0^1
\dfrac{(D(u))^3 S''(u)}{2(S(u))^2} |\partial_xu|^4dx
-
 \int_0^1
\dfrac{(D(u))^2 S''(u)}{2S(u)} |\partial_xu|^2 \partial_xu\partial_x vdx.
$$
When we focus on the case $q\in [0,1]$,  we have $S''(u)\leq 0$.
Thus the first term of the above is a good term and we invoke this term and suitable regularity estimates on $\partial_xv$ to kill the second term.
\subsubsection{Entropy production estimate}\label{section_est}
We recall the regularity estimate on $v$ and $\partial_ x v$.
Due to the mass conservation of $u$,  the standard semigroup estimate guarantees the following.
\begin{lem}\label{reg_v}
Let $(u,v)$ be a solution of \eqref{P1} in $(0,T)\times (0,1)$.
Then for all $r \in [1, \infty)$, there exists some $C=C(T, u_0, v_0, r)$ such that
$$
\|v(t)\|_{L^r} +\|\partial_x v(t)\|_{L^r} \leq C, \qquad t\in (0,T).
$$
\end{lem}
The direct calculations imply the following.
\begin{lem}\label{lem_S}
$S(\tau) = \tau (1+\tau)^{-q}$ satisfies
\begin{eqnarray*}
S^\prime (\tau) &=& (1+\tau )^{-q-1}(1+\tau - q\tau),\\
S^{\prime \prime} (\tau) &=& (1+\tau)^{-q-2} (q(q-1)\tau -2q).
\end{eqnarray*}
Moreover, if $q\in [0,1]$, we have
$$
S^{\prime \prime} (\tau) \leq 0 \qquad (\tau \geq0).
$$
\end{lem}
In light of the fact $S^{\prime \prime} \leq 0$ and Lemma \ref{reg_v} we have the following
lemma, which is an important step in the proof of Proposition \ref{prop_regularity_ep}.
\begin{lem}\label{lem_est_ep1}
Assume that $p-q = 1$ and $q\in \left(\frac{1}{2}, 1\right]$.
Let $(u,v)$ be a solution of \eqref{P1} in $(0,T)\times (0,1)$.
There exists some $C=C(T)$ such that
$$
\mathcal{F}(u(t)) \leq C, \qquad t\in (0, T).
$$
\end{lem}
\begin{pr}{Lemma \ref{lem_est_ep1}}
The identity \eqref{entro_prod} can be written as follows:
\begin{eqnarray}
\dfrac{d}{dt} \mathcal{F}(u(t)) + \mathcal{D}(u(t),v(t)) =I_1+I_2+I_3,
\end{eqnarray}
where
\begin{eqnarray*}
I_1 &:=& \int_0^1 \frac{S(u)D(u)(v+v_t)^2}{4}dx,\\
I_2 &:=& \int_0^1 \frac{(D(u))^3 S''(u)}{2(S(u))^2} |\partial_xu|^4dx,\\
I_3 &:=& -\int_0^1 \frac{(D(u))^2 S''(u)}{2S(u)} |\partial_xu|^2\partial_xu\partial_xvdx.
\end{eqnarray*}
As to $I_1$, we can proceed the same  way as in \cite{Ci-Fu18}.
Since
$$
S(u)D(u) = (1+u)^{-p} \cdot u(1+u)^{-q} \leq (1+u)^{1-p-q},
$$
and the assumptions $p-q = 1$ and $q\in \left(\frac{1}{2}, 1\right]$ yield
$$
1-p-q = -2q <0,
$$
thus $S(u)D(u) \leq 1$.  By Lemma \ref{reg_v}, it follows
$$
|I_1| \leq C + C \int_0^1 |v_t|^2dx
$$
with some $C>0$.

Next, we estimate $I_3$. By Young's inequality, we have
\begin{eqnarray*}
|I_3| &\leq& \frac{1}{2} \int_0^1 \frac{(D(u))^2 |S''(u)|}{S(u)} |\partial_xu|^3  |\partial_xv|dx\\
&=&  \frac{1}{2} \int_0^1 \dfrac{(D(u))^{\frac{9}{4}} |S''(u)|^\frac{3}{4}}{(S(u))^{\frac{3}{2}}} |\partial_xu|^3 \cdot \frac{(S(u))^\frac{1}{2} |S''(u)|^{\frac{1}{4}}}{ (D(u))^\frac{1}{4}}  |\partial_xv|dx\\
&\leq& \frac{1}{2}  \int_0^1 \frac{(D(u))^{3} |S''(u)| }{(S(u))^{2}} |\partial_xu|^4dx
  +C \int_0^1 \frac{(S(u))^2 |S''(u)|}{ D(u)}  |\partial_xv|^4dx
\end{eqnarray*}
with some $C>0$.
Since $S^{\prime \prime} \leq 0$ as noted in Lemma \ref{lem_S}, we obtain
\begin{eqnarray*}
I_2+|I_3| \leq C \int_0^1 \frac{(S(u))^2 |S''(u)|}{ D(u)}  |\partial_xv|^4dx.
\end{eqnarray*}
It follows from Lemma \ref{lem_S} that
\begin{eqnarray*}
\frac{(S(u))^2 |S''(u)|}{ D(u)}
& =& \frac{u^2(1+u)^{-2q} (1+u)^{-q-2} |q(q-1)u -2q| }{(1+u)^{-p}}\\
&=&u^2(1+u)^{p-3q-2} |q(q-1)u -2q|\\
&\leq & C (1+u)^{p-3q+1},
\end{eqnarray*}
and then we invoke the assumption on $(p,q)$ to see
\begin{eqnarray*}
\frac{(S(u))^2 |S''(u)|}{ D(u)}   \leq  C (1+u)^{2(1-q)}.
\end{eqnarray*}
Since $q>\frac{1}{2}$ guarantees
$$
2(1-q) < 1,
$$
and with the mass conservation law and Lemma \ref{reg_v}, we can see that
$$
\int_0^1 \frac{(S(u))^2 |S''(u)|}{ D(u)}  |\partial_xv|^4dx\leq C\int_0^1 (1+u)^{2(1-q)}|\partial_xv|^4dx \leq C.
$$
Finally, combining the above calculations, we have
$$
\frac{d}{dt} \mathcal{F}(u(t)) \leq C+C\int_0^1 |v_t|^2dx.
$$
Since boundedness from below of the classical Lyapunov functional $\mathcal{L}$ (see~\eqref{Lyap_est_KS}) implies
\begin{equation}\label{F2}
\int_0^T \int_0^1 |v_t|^2 dx dt\leq C,
\end{equation}
we conclude the proof.
\end{pr}

We next prove that the term $\int_0^1 \Psi(u)dx$ can be estimated.
\begin{lem}\label{lem_est_ep2}
Assume that $p-q = 1$ and $q\in \left(\frac{1}{2}, 1\right]$.
Let $(u,v)$ be a solution of \eqref{P1} in $(0,T)\times (0,1)$.
There exists some $C=C(p, M)$ such that
$$
\int_0^1 \Psi(u) dx \leq C, \qquad t\in (0, T).
$$
\end{lem}
\begin{pr}{Lemma \ref{lem_est_ep2}}
Since Lemma \ref{lem_S} implies
$$
\frac{\tau D(\tau) S'(\tau)}{S(\tau)} = \frac{\tau (1+\tau)^{-p} (1+\tau )^{-q-1}(1+\tau - q\tau)}{\tau (1+\tau)^{-q}}
= (1+\tau )^{-p-1}(1+(1-q)\tau),
$$
it follows from $q\leq 1$ that
$$
\frac{\tau D(\tau) S'(\tau)}{S(\tau)} \leq (1+\tau )^{-p}, \qquad \tau\geq 0.
$$
Hence, we have that for all $r\geq 0$,
\begin{eqnarray*}
\int_1^r \frac{\tau D(\tau) S'(\tau)}{S(\tau)}d\tau  \leq  \int_1^r (1+\tau )^{-p} d\tau
= \frac{2^{1-p}}{p-1} - \frac{(1+r)^{1-p}}{p-1} \leq \frac{2^{1-p}}{p-1},
\end{eqnarray*}
where we remark that the assumptions $p-q = 1$ and $q\in \left(\frac{1}{2}, 1\right]$ imply $p \in(\frac{3}{2},2]$.
Therefore, there exists some $C>0$ such that for all $\phi\geq0$,
\begin{eqnarray*}
\Psi(\phi) &\leq & \int_1^\phi \left( \frac{2^{1-p}}{p-1} + r D(r) \right)dr\\
&= & \int_1^\phi \left( \frac{2^{1-p}}{p-1} + r(1+r)^{-p} \right)dr\\
&\leq & \int_1^\phi \left( \frac{2^{1-p}}{p-1} + (1+r)^{-p+1} \right)dr\\
&\leq&
\begin{cases} C(\phi+(1+\phi)^{-p+2})\qquad (p<2),\\
C(\phi+\log (1+\phi))\qquad (p=2).
\end{cases}
\end{eqnarray*}
Since we have $0\leq -p+2 <1$,  the mass conservation law implies
$$
\int_0^1 \Psi(u) dx\leq C(p, M).
$$
We conclude the proof.
\end{pr}
\vspace{2mm}

We are in a position to prove Proposition \ref{prop_regularity_ep}.
\vspace{2mm}

\begin{pr}{Proposition \ref{prop_regularity_ep}}
Lemma \ref{lem_est_ep1} and Lemma \ref{lem_est_ep2} conclude the proof.
\end{pr}

\vspace{2mm}
Next, we derive the a priori estimate helpful in the proof of Theorem \ref{MT}.
\begin{lem}\label{lem_est_Lp}
Assume that $p-q = 1$ and $q\in \left(\frac{1}{2}, 1\right]$.
Let $(u,v)$ be a solution of \eqref{P1} in $(0,T)\times (0,1)$.
It holds
\begin{align*}
\frac{d}{dt} \int_0^1 u^pdx
&+ p(p-1)\int_0^1 u^{p-2}(1+u)^{-p} |\partial_xu|^2dx
\\
\leq\,&  \frac{3p(p-1)}{2} \int_0^1  u^2 dx
+ \frac{p(p-1)}{2} \int_0^1  |v_t|^2dx.
\end{align*}
\end{lem}
\begin{pr}{Lemma \ref{lem_est_Lp}}
Multiplying the first equation of \eqref{P1} by $pu^{p-1}$ and using the integration by parts, we have
\begin{align*}
\frac{d}{dt} \int_0^1 u^pdx =\,& p \int_0^1 u^{p-1}u_tdx\\
=\,& p \int_0^1 u^{p-1} \partial_x ((1+u)^{-p}\partial_xu - u(1+u)^{1-p}\partial_xv)dx\\
=\,& -p(p-1)\int_0^1 u^{p-2}(1+u)^{-p} |\partial_xu|^2dx
\\&+ p(p-1) \int_0^1 u^{p-1} (1+u)^{1-p} \partial_xu\partial_xvdx,
\end{align*}
where we used $-q = 1-p$.
We next use the second equation of \eqref{P1} and invoke the integration by parts again to have
\begin{eqnarray*}
\frac{d}{dt} \int_0^1 u^pdx
&+& p(p-1)\int_0^1 u^{p-2}(1+u)^{-p} |\partial_xu|^2dx\\
&=& p(p-1) \int_0^1 \left(\frac{u}{1+u}\right)^{p-1}\partial_xu\partial_x vdx\\
&=& p(p-1) \int_0^1 \partial_x \left( \int_0^u  \left(\frac{s}{1+s}\right)^{p-1}\,ds \right) \partial_x vdx\\
&=& - p(p-1) \int_0^1  \left( \int_0^u  \left(\frac{s}{1+s}\right)^{p-1}ds \right)  \partial_{xx} vdx\\
&=& - p(p-1) \int_0^1  \left( \int_0^u  \left(\frac{s}{1+s}\right)^{p-1}ds \right)  (v_t + v -u)dx.
\end{eqnarray*}
Since $p \in (\frac{3}{2}, 2]$ implies
$$
0 \leq \left(\frac{s}{1+s}\right)^{p-1} \leq 1,
$$
it follows
$$
0 \leq \int_0^u  \left(\frac{s}{1+s}\right)^{p-1}ds
\leq u.
$$
Combing the above estimates and using nonnegativity of $v$, we have
\begin{eqnarray*}
\frac{d}{dt} \int_0^1 u^pdx
&+& p(p-1)\int_0^1 u^{p-2}(1+u)^{-p} |\partial_xu|^2dx\\
&\leq &  p(p-1) \int_0^1  u |v_t| dx+ p(p-1) \int_0^1  u^2dx\\
&\leq & \frac{3p(p-1)}{2} \int_0^1  u^2dx + \frac{p(p-1)}{2} \int_0^1  |v_t|^2dx.
\end{eqnarray*}
This is the desired inequality.
\end{pr}
\begin{lem}\label{lem_est_L2}
Assume that $(p,q) = (2,1)$.
Let $(u,v)$ be a solution of \eqref{P1} in $(0,T)\times (0,1)$.
There exists some $C=C(T)$ such that
\begin{eqnarray*}
\int_0^1 u^pdx \leq  C.
\end{eqnarray*}
\end{lem}
\begin{proof}
By Lemma \ref{lem_est_Lp}, we have
\begin{eqnarray*}
\frac{d}{dt} \int_0^1 u^pdx
\leq  C_1 \int_0^1  u^pdx + C_2 \int_0^1  |v_t|^2dx.
\end{eqnarray*}
We invoke the Gronwall inequality and \eqref{F2} to conclude.
\end{proof}
Now with the above auxiliary regularity estimate,  we give the proof of Theorem \ref{MT}.

\begin{pr}{Theorem \ref{MT}}
It follows from the one-dimensional Sobolev embedding that
\begin{align*}
\|\log (1+u)\|_{L^\infty} \le\,& C \|\log (1+u)\|_{W^{1,1}}
\\
\le\, & C(M+1) + \int_0^1 |\partial_x (\log (1+u))|dx
\\
=\,& C(M+1) + \int_0^1 \frac{|\partial_xu|}{1+u}dx
\\
\le\, & C(M+1)\\&+\left(\int_0^1 (1+u)^pdx \right)^{\frac{1}{2}} \left(\int_0^1 \frac{|\partial_xu|^2}{(1+u)^p(1+u)^2}dx \right)^{\frac{1}{2}}.
\end{align*}
When $(p,q) = (2,1)$, Lemma \ref{lem_est_L2} guarantees
$$
\int_0^1 (1+u)^p dx \leq C(T).
$$
Moreover, Proposition \ref{prop_regularity_ep} implies
$$
\int_0^1 \frac{|\partial_xu|^2}{(1+u)^p(1+u)^2}dx
 \leq \int_0^1 \frac{|\partial_xu|^2}{u(1+u)^{p+1}}dx \leq C(T).
$$
Combining the above implies
$$
\|\log (1+u)\|_{L^\infty} \leq C(T),
$$
and then global existence is established.
\end{pr}
\section{The case of a $p$-Laplace equation}\label{piata}

We consider the initial boundary problem for the $p$-Laplace equation equipped with $0$-Neumann boundary condition:
\begin{equation}
\left\{
\begin{aligned}
& u_t = \nabla\cdot(|\nabla u|^{p-2}\nabla u),
& t>0,\, &x\in\Omega,
\\
&\frac{\partial u}{\partial \nu}=0,
&t>0,\, &x\in\partial\Omega,
\\
&u(0,x)=u_{0}(x),
&\, &x\in\Omega,
\end{aligned}
\right.
\label{eqn;pLE}
\end{equation}
where $\Omega$ is a bounded domain of $\r^n$ $(n\geq 1)$ with smooth boundary $\partial\Omega$,
the initial data
$u_0\in W^{1,\infty}(\Omega)$ such that $u_0\geq 0$ in $\Omega$, and $1\le p <\infty$.
Our estimates in this section are formal. We assume the regular solution.
The main result of this section is a seemingly new Lyapunov functional for the 1D $p$-Laplace equation.
On the other hand, we also give an a priori estimate in dimensions $n\ge2$, which might be of interest.
We remark that our new functional $I$ with $p=2$ corresponds to the Fisher information for the linear heat equation.

\subsection{A priori estimate in higher dimensions}
\begin{prop}\label{prop;pLaplace_higher}
	Let $u$ be the positive classical solution to \eqref{eqn;pLE}.
	Set $I[u]$ as
	\begin{align*}
	I[u]:=\int_{\Omega}|\nabla u^{1-\frac1{2(p-1)}}|^{p}  dx.
	\end{align*}
	Then for $p\neq 3/2$
	\begin{align*}
	&	\frac{d}{dt}I[u]\\=\,&-pp_*^{2-p}
	\int_{\Omega}\left|u^{\frac14-\frac1{4(p-1)}}\nabla\cdot\left(|\nabla u^{1-\frac1{2(p-1)}}|^{p-2} \nabla u^{1-\frac1{2(p-1)}}\right)\right|^2 dx
	\\
	&+\frac{p^2}{2}p_*^{1-p}
	\int_{\Omega} u^{-\frac{1}{2}} |\nabla u^{1-\frac1{2(p-1)}}|^{2p-4} (\nabla u^{1-\frac1{2(p-1)}})^TD^2 u^{1-\frac1{2(p-1)}}\nabla u^{1-\frac1{2(p-1)}}dx
	\\
	&-\frac {p}{4}p_*^{-p}
	\int_{\Omega}\left||\nabla u^{1-\frac1{2(p-1)}}|^{p} u^{-\frac{3}{4}+\frac{1}{4(p-1)}}\right|^{2} dx,
	\end{align*}
where
	\begin{align*}
	p_*:=1-\frac1{2(p-1)}.
	\end{align*}
\end{prop}
\vspace{3mm}
\begin{pr}{Proposition \ref{prop;pLaplace_higher}}
	Differentiating $I$ with respect to time,
	\begin{align}
	\frac{d}{dt} \int_{\Omega} |\nabla u^{1-\frac1{2(p-1)}} |^{p}  dx
	=\,&
	p\int_{\Omega} |\nabla u^{1-\frac1{2(p-1)}} |^{p-1} \frac{\nabla u^{1-\frac1{2(p-1)}} }{|\nabla u^{1-\frac1{2(p-1)}} |} \cdot\nabla\partial_t u^{1-\frac1{2(p-1)}}   dx
	\nonumber\\
	=\,&
	p\int_{\Omega}|\nabla u^{1-\frac1{2(p-1)}}  |^{p-2}\nabla u^{1-\frac1{2(p-1)}}\cdot \nabla\partial_t u^{1-\frac1{2(p-1)}}  dx.
	\label{eqn;p-1}
	\end{align}
	Since
	\begin{align*}
	\partial_t u^{1-\frac1{2(p-1)}}
	=\,&\left(1-\frac{1}{2(p-1)}\right) u^{-\frac1{2(p-1)}}\partial_t u
	\\
	=\,&
	\left(1-\frac1{2(p-1)}\right) u^{-\frac1{2(p-1)}}\nabla\cdot (|\nabla u|^{p-2} \nabla u),
	\end{align*}
	the right hand side of \eqref{eqn;p-1} is represented as
	\begin{align*}
	&p\int_{\Omega}|\nabla u^{1-\frac1{2(p-1)}}  |^{p-2}\nabla u^{1-\frac1{2(p-1)}}\cdot \nabla\partial_t u^{1-\frac1{2(p-1)}}  dx
	\\
	=\,&pp_*
	\int_{\Omega}|\nabla u^{1-\frac1{2(p-1)}}|^{p-2}\nabla u^{1-\frac1{2(p-1)}}\cdot\nabla\left[u^{-\frac1{2(p-1)}}\nabla\cdot (|\nabla u|^{p-2} \nabla u)\right] dx,
			\eqntag
	\label{eqn;p-0}
	\end{align*}
		where
	\begin{align*}
	p_*:=1-\frac1{2(p-1)}.
	\end{align*}	
	Moreover, noting
	\begin{align*}
	\nabla u = \left(1-\frac{1}{2(p-1)}\right)^{-1} u^{\frac{1}{2(p-1)}} \nabla u^{1-\frac1{2(p-1)}},
	\end{align*}
	we see
	\begin{align*}
	|\nabla u|^{p-2} \nabla u=\left(1-\frac{1}{2(p-1)}\right)^{1-p} u^{\frac{1}{2}} |\nabla u^{1-\frac1{2(p-1)}} |^{p-2} \nabla u^{1-\frac{1}{2(p-1)}},
\eqntag \label{eqn;p-2}
	\end{align*}
	so that by using \eqref{eqn;p-2} in \eqref{eqn;p-0} and integrating by parts,
	\begin{align}
	&pp_*\int_{\Omega}|\nabla u^{1-\frac{1}{2(p-1)}}|^{p-2}\nabla u^{1-\frac1{2(p-1)}}\cdot\nabla\left[u^{-\frac1{2(p-1)}}\nabla\cdot (|\nabla u|^{p-2} \nabla u)\right] dx
	\nonumber\\
	=\,&pp_*^{2-p}
	\int_{\Omega}|\nabla u^{p_*}  |^{p-2}\nabla u^{p_*}
	\cdot\nabla\left[u^{-\frac1{2(p-1)}}\nabla\cdot \left(u^{\frac{1}{2}} |\nabla u^{p_*}|^{p-2} \nabla u^{p_*}\right)\right] dx
	\nonumber\\
	=\,&pp_*^{2-p}
	\int_{\Omega}|\nabla u^{p_*}|^{p-2}\nabla u^{p_*}\cdot
	\nabla\left[u^{-\frac1{2(p-1)}} |\nabla u^{p_*} |^{p-2} \nabla u^{\frac{1}{2}}\cdot\nabla u^{p_*} \right]dx
	\nonumber\\
	&+pp_*^{2-p}
	\int_{\Omega}|\nabla u^{p_*} |^{p-2}\nabla u^{p_*}\cdot
	\nabla\left[u^{\frac{1}{2}-\frac{1}{2(p-1)}}\nabla\cdot \left( |\nabla u^{p_*} |^{p-2} \nabla u^{p_*}\right)\right] dx
	\nonumber\\
	=\,&pp_*^{2-p}
	\int_{\Omega}|\nabla u^{p_*}|^{p-2}\nabla u^{p_*}\cdot
	\nabla\left[u^{-\frac{1}{2(p-1)}} |\nabla u^{p_*} |^{p-2} \nabla u^{\frac{1}{2}}\cdot\nabla u^{p_*}\right] dx
	\nonumber\\
	&-pp_*^{2-p}
	\int_{\Omega}
	u^{\frac{1}{2}-\frac{1}{2(p-1)}} \left|\nabla\cdot \left( |\nabla u^{p_*}  |^{p-2} \nabla u^{p_*} \right)\right|^2 dx,
	\label{eqn;p-3}
	\end{align}
	where the $0$-Neumann boundary condition for $u$ implies 
	\begin{align*}
	\nabla u^{p_*}\cdot \nu =\, p_* u^{p_*-1} \nabla u \cdot \nu =0\quad\text{on}~\partial\Omega.
	\end{align*}
Here, as $p\neq 3/2$, then
	\begin{align*}
	u^{-\frac1{2(p-1)}}\nabla u^{\frac{1}{2}}=\frac{1}{2p_*} u^{-\frac{1}{2}} \nabla u^{p_*},
	\end{align*}
	and then the first term on the right hand side of \eqref{eqn;p-3} is represented as
	\begin{align*}
	&p p_*^{2-p}
	\int_{\Omega}|\nabla u^{p_*}|^{p-2}\nabla u^{p_*}
	\cdot\nabla\left[u^{-\frac{1}{2(p-1)}} |\nabla u^{p_*} |^{p-2} \nabla u^{\frac{1}{2}}\cdot\nabla u^{p_*}\right]dx
	\\
	=\,&\frac {p}{2}p_*^{1-p}
	\int_{\Omega}|\nabla u^{p_*}|^{p-2}\nabla u^{p_*}\cdot
	\nabla\left[u^{-\frac{1}{2}} |\nabla u^{p_*} |^{p} \right] dx.
	\end{align*}
	Since
	\begin{align*}
	&	\nabla\left[u^{-\frac{1}{2}}|\nabla u^{p_*}|^{p}\right]
	\\
	=\,&\nabla u^{-\frac{1}{2}}  |\nabla u^{p_*}|^{p}
	+u^{-\frac1{2}} \nabla|\nabla u^{p_*}|^{p}
	\\
	=\,&-\frac{1}{2} u^{-\frac{3}{2}} \nabla u |\nabla u^{p_*} |^{p}
	+u^{-\frac{1}{2}} p|\nabla u^{p_*} |^{p-2}\nabla u^{p_*}D^2 u^{p_*}
	\\
	=\,&
	-\frac{1}{2}p_*^{-1} u^{-\frac{3}{2}+\frac{1}{2(p-1)}} \nabla u^{p_*} |\nabla u^{p_*} |^{p}
	+u^{-\frac{1}{2}} p|\nabla u^{p_*} |^{p-2}\nabla u^{p_*}D^2 u^{p_*},
	\end{align*}
	one notices
	\begin{align*}
	&\frac {p}{2}p_*^{1-p}
\int_{\Omega}|\nabla u^{p_*}|^{p-2}
\nabla u^{p_*}\cdot\nabla\left[u^{-\frac{1}{2}} |\nabla u^{p_*} |^{p} \right]dx
\\=\,&-\frac{p}{4}p_*^{-p}
\int_{\Omega} u^{-\frac{3}{2}+\frac{1}{2(p-1)}} |\nabla u^{p_*} |^{2p} dx
\\
&+\frac{p^2}{2}p_*^{1-p}
\int_{\Omega}
u^{-\frac{1}{2}} |\nabla u^{p_*}  |^{2p-4} (\nabla u^{p_*})^T  D^2 u^{p_*} \nabla u^{p_*}  dx.
	\end{align*}
	Combining above with \eqref{eqn;p-3} and plugging the outcome into \eqref{eqn;p-0},
	we realize
	\begin{align*}
	\frac{d}{dt}I[u]
	=\,&
-	pp_*^{2-p}
	\int_{\Omega}
	u^{\frac{1}{2}-\frac{1}{2(p-1)}} \left|\nabla\cdot \left( |\nabla u^{p_*}  |^{p-2} \nabla u^{p_*} \right)\right|^2 dx
	\\
	&+\frac{p^2}{2}p_*^{1-p}\int_{\Omega}
	u^{-\frac{1}{2}} |\nabla u^{p_*}  |^{2p-4} (\nabla u^{p_*})^T  D^2 u^{p_*} \nabla u^{p_*} dx
	\\
	&-\frac{p}{4}p_*^{-p}
	\int_{\Omega} u^{-\frac{3}{2}+\frac{1}{2(p-1)}} |\nabla u^{p_*} |^{2p} dx,
	\end{align*}
	where $p_*:=1-\frac{1}{2(p-1)}$,
	as claimed.
\end{pr}

\vspace{3mm}
\subsection{Lyapunov functional in 1D $p$-Laplace equations}

Let us consider the solution to \eqref{eqn;pLE}.
We shall show that $I[u]$ defined in Proposition \ref{prop;pLaplace_higher}
decreases along the trajectories of \eqref{eqn;pLE} for $p\ge2$.
Indeed, we have the following theorem.

\begin{theorem}\label{thm;p-Laplace-1d}
Assume $p\ge2$. Let $u$ be a solution to \eqref{eqn;pLE} in $\Omega=(0,1)$.
Then,
\begin{align*}
I[u]:=\int_0^1 | \partial_x (u^{1-\frac{1}{2(p-1) } })|^pdx
\end{align*}
 is non-increasing in time.
\end{theorem}

\begin{pr}{Theorem \ref{thm;p-Laplace-1d}}
Let us recall $p_*:=1-\frac{1}{2(p-1)}$.
By Proposition~\ref{prop;pLaplace_higher}, we know that
	\begin{align*}
\frac{d}{dt}I[u]=\,&-pp_*^{2-p}
\int_{0}^1\left|u^{\frac14-\frac1{4(p-1)}}\partial_x\left(|\partial_x u^{p_*}|^{p-2} \partial_x u^{p_*}\right)\right|^2 dx
\\
&+\frac{p^2}{2}p_*^{1-p}\int_{0}^1 u^{-\frac{1}{2}} |\partial_x u^{p_*}|^{2p-4} \partial_xu^{p_*} \partial_{xx} u^{p_*}\partial_x u^{p_*} dx
\\
&-\frac {p}{4}p_*^{-p}
\int_{0}^1\left||\partial_x u^{p_*}|^{p} u^{-\frac{3}{4}+\frac{1}{4(p-1)}}\right|^{2} dx
	\\
	=\,&
	-pp_*^{-p}
	\int_{0}^1
	\left|p_*u^{\frac{1}{4}-\frac{1}{4(p-1)}} \partial_x\left( |\partial_x u^{p_*}  |^{p-2} \dx u^{p_*} \right)\right|^2 dx
	\\
	&+pp_*^{-p}\int_{0}^1\frac {pp_*}{2}
	u^{-\frac{1}{2}} |\partial_x u^{p_*}  |^{2p-2} \partial_{xx} u^{p_*}  dx
\\
	&-pp_*^{-p}
	\int_{0}^1 \left| \frac{1}{2}u^{-\frac{3}{4}+\frac{1}{4(p-1)}} |\partial_x u^{p_*} |^{p} \right|^2 dx.
\end{align*}
	Moreover,
	\begin{align*}
		&\partial_x\left(|\partial_x u^{p_*} |^{p-2} \partial_x u^{p_*}\right)
		\\
		=\,&(p-2)|\partial_x u^{p_*}|^{p-4}\partial_x u^{p_*} \partial_{xx} u^{p_*} \partial_x u^{p_*}
	+|\partial_x u^{p_*}|^{p-2} \partial_{xx} u^{p_*}
	\\
	=\,&(p-1)|\partial_x u^{p_*}|^{p-2}
	 \partial_{xx} u^{p_*} ,
		\end{align*}
consequently,
		\begin{align*}
		&p^{-1}p_*^p\frac{d}{dt}I[u]
		\\
		=\,&
		-\int_{0}^1\left|p_*u^{\frac{1}{4}-\frac{1}{4(p-1)}}\partial_x\left(|\partial_x u^{p_*}|^{p-2} \partial_x u^{p_*}\right)\right|^2 dx
	\\
		&+
	\int_{0}^1 \frac {pp_*}{2}u^{-\frac{1}{2}} |\partial_x u^{p_*}|^{2p-2}
		\partial_{xx} u^{p_*}
		 dx
		-\int_{0}^1\left|\frac{1}{2}|\partial_x u^{p_*}|^{p}  u^{-\frac{3}{4}+\frac{1}{4(p-1)}}\right|^{2} dx
		\\
		=\,&
		-\int_{0}^1\left|p_*u^{\frac{1}{4}-\frac{1}{4(p-1)}}(p-1)|\partial_x u^{p_*}|^{p-2}
\partial_{xx}u^{p_*}
\right|^2 dx
	\\
		&+\frac{p}{2}
		\int_{0}^1 p_*u^{-\frac{1}{2}} |\partial_x u^{p_*}|^{2p-2}
\partial_{xx} u^{p_*}
 dx
		\\
		&-\left(\frac {p}{2(p-1)}\right)^2
		\int_{0}^1\left|\frac{1}{2}|\partial_x u^{p_*}|^{p} u^{-\frac{3}{4}+\frac{1}{4(p-1)}}\right|^{2}dx
		\\
		&-\left(1-\left(\frac {p}{2(p-1)}\right)^2\right)
		\int_{\Omega}\left|\frac{1}{2}|\partial_xu^{p_*}|^{p} u^{-\frac{3}{4}+\frac{1}{4(p-1)}}\right|^{2} dx
		\\
		=\,&
		-\int_{0}^1|\partial_x u^{p_*}|^{2p}\Bigg|p_*(p-1)u^{\frac{1}{4}-\frac{1}{4(p-1)}}|\partial_x u^{p_*}|^{-2}
\partial_{xx} u^{p_*}
	-\frac{p}{4(p-1)}u^{-\frac{3}{4}+\frac{1}{4(p-1)}}\Bigg|^2  dx
		\\
		&-\left(1-\left(\frac {p}{2(p-1)}\right)^2\right)
		\int_{0}^1\left|\frac12|\partial_x u^{p_*}|^{p} u^{-\frac{3}{4}+\frac{1}{4(p-1)}}\right|^{2}\ dx
		\\
		\le\,&0,
		\end{align*}
		provided
		\begin{align*}
		p>1,\quad p\neq \frac{3}{2},\quad
	p_*> 0,
		\quad
		1-\left(\frac {p}{2(p-1)}\right)^2 \geq 0,
		\end{align*}
		 hence
		\begin{align*}
		p\geq  2.
		\end{align*}
		Therefore, if $2\leq  p$, then
		\begin{align*}
		\frac{d}{dt}I[u] \leq 0,
		\end{align*}
		which concludes the proof.
\end{pr}

\vspace{5mm}
\noindent
{\bf Acknowledgments.}
The work of the second author is supported by JSPS KAKENHI Grant Numbers 24H00184,  25K21996.
The work of the third author is partially supported by JSPS Early-Career Scientists, Grant Number 25K17274 and JSPS Fellows, Grant Number 25KJ0279.

\end{document}